\documentclass[11pt]{amsart}
\usepackage[margin=1in]{geometry}
\usepackage{amssymb,mathtools,tikz,booktabs,microtype,fancyvrb,placeins,needspace,enumitem}
\usepackage[dvipsnames]{xcolor}
\usepackage[numbers,sort&compress]{natbib}
\usepackage[colorlinks=true,linkcolor=MidnightBlue,citecolor=MidnightBlue,urlcolor=MidnightBlue]{hyperref}

\theoremstyle{plain}
\newtheorem{theorem}{Theorem}[section]
\newtheorem{theoremalpha}{Theorem}

\newtheorem{proposition}[theorem]{Proposition}
\newtheorem{lemma}[theorem]{Lemma}
\newtheorem{corollary}[theorem]{Corollary}
\theoremstyle{definition}

\newtheorem{example}[theorem]{Example}
\newtheorem{remark}[theorem]{Remark}

\usetikzlibrary{arrows.meta}
\definecolor{cutblue}{RGB}{31,84,140}
\definecolor{cutorange}{RGB}{201,110,26}
\definecolor{cutteal}{RGB}{17,122,110}

\DeclareMathOperator{\rank}{rank}
\DeclareMathOperator{\spn}{span}
\DeclareMathOperator{\Diag}{Diag}
\DeclareMathOperator{\diag}{diag}

\DeclareMathOperator{\range}{range}
\DeclareMathOperator{\cut}{cut}
\DeclareMathOperator{\SDP}{SDP}
\DeclareMathOperator{\MC}{MC}
\DeclareMathOperator{\conv}{conv}

\newcommand{\Ell}{\mathcal E_n}

\newcommand{\sgn}{\{-1,1\}}
\newcommand{\one}{\mathbf1}
\newcommand{\Keff}{K_{\mathrm{eff}}}
\newcommand{\F}{F^*(W)}
\newcommand{\Pcut}{P^*(W)}
\DeclareMathOperator{\tr}{tr}
\DeclareMathOperator{\bit}{bit}
\newcommand{\SDF}{\mathsf{SDF}}
\newcommand{\CH}{\mathsf{CH}}
\newcommand{\CHone}{\mathsf{CH}_1}
\newcommand{\NP}{\mathsf{NP}}
\newcommand{\coNP}{\mathsf{coNP}}
\newcommand{\PP}{\mathsf P}
\DeclareMathOperator{\relint}{relint}

\title[Optimal Cuts and Optimal Faces of Exact Max-Cut SDPs]{Optimal Cuts and Optimal Faces of Exact Max-Cut Semidefinite Relaxations}
\author{Avinash Bhardwaj}
\address{Department of Industrial Engineering and Operations Research, Indian Institute of Technology Bombay, Mumbai 400076, India}
\email{abhardwaj@iitb.ac.in}
\author{Chen Chen}
\address{Department of Integrated Systems Engineering, The Ohio State University, 1971 Neil Avenue, Columbus, OH 43210, USA}
\email{chen.8018@osu.edu}
\author{Vishnu Narayanan}
\address{Department of Industrial Engineering and Operations Research, Indian Institute of Technology Bombay, Mumbai 400076, India}
\email{vishnu@iitb.ac.in}
\hypersetup{pdftitle={Optimal Cuts and Optimal Faces of Exact Max-Cut Semidefinite Relaxations},pdfauthor={Avinash Bhardwaj, Chen Chen, Vishnu Narayanan}}

\begin{document}

\begin{abstract}
An exact Max-Cut semidefinite relaxation can have optima outside the convex hull of optimal cut matrices. We study when this hull equals the optimal face and how to recognize equality. For rational weights, the unique optimal dual slack is recoverable in polynomial time, but its kernel may contain directions no optimal matrix uses. Under strict complementarity, recognition and recovery of all optimal cuts on equality instances take deterministic polynomial time. At dual rank one, this hypothesis is unnecessary. Recognition is polynomial-time Turing equivalent to exact rational semidefinite feasibility, even on simple unweighted graphs. Testing whether a supplied cut matrix is the unique optimum is many-one equivalent, even on connected simple graphs with a unique optimal cut. NP-completeness of recognition would imply NP equals coNP. Replication realizes doubly nonnegative matrices with positive diagonal and rational kernel as sign-spanned optimal dual slacks, preserving rank, complete positivity, and cp-rank.
\end{abstract}

\maketitle

\section{Introduction}
The Maximum Cut problem asks for a partition of the vertices of an edge-weighted graph that maximizes the weight of the edges crossing the partition. For a symmetric nonnegative weight matrix $W$ with zero diagonal, its sign formulation and semidefinite relaxation are
\[
\MC(W)=\max_{x\in\sgn^n}\frac14\sum_{i\ne j}W_{ij}(1-x_ix_j),
\qquad
\SDP(W)=\max_{X\in\Ell}\frac14\sum_{i\ne j}W_{ij}(1-X_{ij}),
\]
where $\Ell=\{X\in\mathbb S^n:X\succeq0,\ \diag X=\one\}$ is the elliptope. Every cut gives a feasible rank-one matrix $xx^\top$, so $\MC(W)\le\SDP(W)$. This relaxation underlies the approximation algorithm of Goemans and Williamson \cite{goemans1995improved}. When equality holds, the relaxation is called exact.

Exactness guarantees a rank-one optimum, but an exact instance can also have optimal matrices outside the convex hull of its optimal cut matrices \cite{Bhardwaj2X}. Such a matrix cannot be interpreted as the second moment $\mathbb E[xx^\top]$ of any probability distribution on optimal cuts. We ask when this interpretation holds for every optimum, and how to recognize that property from the graph weights.

The dual certificate describes where all these optima can lie. Let $\delta_i=\sum_jW_{ij}$ and $L_W=\Diag(\delta)-W$. In our normalization, the dual problem is
\[
\operatorname{D}(W)=\min_{s\in\mathbb R^n}\{\one^\top s:Z(s):=\Diag(s)-\tfrac14L_W\succeq0\}.
\]
Strong duality holds. For an optimal dual slack $Z^*$, complementary slackness says that a primal feasible matrix $X$ is optimal precisely when $Z^*X=0$. Thus every optimal matrix has its range in $\ker Z^*$, and every optimal cut sign vector belongs to that kernel. We call an exact instance \emph{sign-spanned} when its optimal cut vectors span the whole kernel.

This gives two ways for the cuts to determine the continuous optimum. Their vectors may span the dual kernel, and their rank-one matrices may generate the optimal face. The second property says that the face equals the \emph{cut hull}, the convex hull of the optimal cut matrices. The distinction matters even for complete graphs: both $K_4$ and $K_6$ are sign-spanned, but only the former has its optimal face generated by cuts.

Our main computational result classifies recognition of cut-hull equality. Recovering the dual certificate is a first step, but its kernel can be larger than the subspace actually used by optimal matrices. We call the latter the \emph{effective primal support}. Strict complementarity identifies it with the dual kernel, making recognition tractable. Without that hypothesis, recognition has the full difficulty of exact SDP feasibility, even on simple unweighted graphs when the unique optimal cut is already known.

\begin{theoremalpha}[Recognition; Proposition~\ref{prop:slackrecovery}, Corollary~\ref{cor:strict} and Theorem~\ref{thm:sdp-equivalence}]\label{thm:intro-recognition}
Given only a rational nonnegative weight matrix, promised to define an exact instance, recognition of equality between the optimal face and its cut hull has the following properties.
\begin{enumerate}
\item The unique optimal dual slack can be recovered in deterministic polynomial time. Under strict complementarity, recognition and recovery of all optimal cut matrices on equality instances take deterministic polynomial time.
\item Recognition is polynomial-time Turing equivalent to exact rational semidefinite feasibility, already for simple unweighted graphs. With a supplied unique optimal cut, the equivalence is many-one, even for connected simple graphs. It also holds for complete graphs with polynomially bounded positive integer weights and a supplied nonnegative integer Gram factor of the slack.
\item Recognition belongs to $\mathsf{NP}$ if and only if it belongs to $\mathsf{coNP}$, and either membership is equivalent to the corresponding membership of exact SDP feasibility. In particular, NP-completeness would imply $\mathsf{NP}=\mathsf{coNP}$.
\end{enumerate}
\end{theoremalpha}
Given a rational nonnegative weight matrix $W$ and a cut sign vector $x$, deciding whether $xx^\top$ is the unique optimal solution of $\SDP(W)$ is polynomial-time many-one equivalent to exact rational semidefinite feasibility; no exactness or optimality promise is required (Corollary~\ref{cor:uniqueness}). This remains true for connected simple unweighted graphs in which the supplied cut is the unique optimal cut. Thus the difficulty persists with no numerical weights to encode: even with an optimal primal--dual pair explicitly available and the dual optimum unique, testing primal uniqueness retains the full difficulty of exact semidefinite feasibility.

The tractable case combines two algorithms: exact recovery of the dual slack by an error bound and rational rounding, and deterministic recognition of a Schur-independent sign basis. Kueng and Tropp recover such sign factors by a randomized algorithm under a factorization promise \cite[Theorem~I and Section~2.4]{KT19II}; our algorithm also rejects subspaces with no such basis and has polynomial bit complexity.

Our lower reduction has two construction steps. The unique-cut embedding (Lemma~\ref{lem:embedding}) places a fitting problem beside a prescribed cut using a nonnegative integer Gram certificate. Explicit left inverses and controlled fitting coefficients keep the weights polynomially bounded. A regular bipartite replacement then removes those weights while preserving the entire optimal face, including every optimal cut (Lemma~\ref{lem:unweight}). The encoding of general SDPs uses Basu--Roy sampling and classical repeated squaring. Conversely, exact SDP queries suffice for recognition, and Ramana's exact alternative \cite{Ramana97} transfers the NP/coNP symmetry of feasibility.

An explicit square-root-sum family in Appendix~\ref{ec:arithmetic} illustrates the obstruction with a linear vertex count: the known cut remains unique while additional optimal matrices appear at an arithmetic threshold. The same threshold determines strict complementarity. Square-root-sum comparison is a classical obstacle in exact semidefinite computation \cite{Goemans98}. Tarasov and Vyalyi reduce arithmetic-circuit comparisons to exact SDP feasibility \cite{TarasovVyalyi08}, while Allender et al. relate square-root-sum comparison to PosSLP and place both problems in the counting hierarchy \cite{AllenderEtAl09}. Porkolab and Khachiyan give bounds on the magnitude of feasible solutions to rational semidefinite systems \cite{PorkolabKhachiyan97}; such bounds can require exponentially many bits to write in binary, motivating the repeated-squaring scale encoding in Lemma~\ref{lem:bounded}. Our reduction preserves the difficulty of general exact feasibility within the optimal face of an exact Max-Cut instance.

The spanning question explains how broad the strictly complementary setting can be. If optimal cut vectors span the dual kernel, averaging their cut matrices produces an optimum of full rank on that kernel. We therefore ask which certificates admit this spanning property. For nonnegative weights a dual slack is doubly nonnegative, meaning positive semidefinite and entrywise nonnegative. The following theorem realizes a broad class of such matrices after replication. Write $J_m$ for the all-ones matrix of order $m$, and call a subspace rational if it has a basis of rational vectors. The threshold $\kappa(Z)$ is the least positive integer $k$ for which integer vectors in $\ker Z$ with coordinate magnitudes at most $k$ span that kernel.

\begin{theoremalpha}[Realization; Theorem~\ref{thm:realization}]\label{thm:intro-realization}
Let $Z$ be doubly nonnegative with positive diagonal and rational kernel. For every integer $k\ge\kappa(Z)$, the matrix $Z\otimes J_{2k}$ is the optimal dual slack of a nonnegatively weighted exact Max-Cut instance whose optimal cut vectors span its kernel. Replication preserves rank, complete positivity, and cp-rank when finite.
\end{theoremalpha}

A matrix is completely positive if it admits a Gram factor with nonnegative entries; its cp-rank is the minimum number of rows in such a factor. The realized slack contains $Z$ as a principal submatrix, so Theorem~\ref{thm:intro-realization} embeds the prescribed matrix while retaining these factorization properties. In particular, sign-spanning does not force complete positivity at any rank at least three. We construct connected examples at every such rank, including an explicit fourteen-vertex instance of rank three. This contrasts with ranks one and two, where doubly nonnegative matrices have nonnegative Gram factors of minimum dimension independently of sign-spanning \cite{berman2003completely,BrandtsKrizek16}.

The arithmetic hypothesis also has an exact explanation. Replication allows signs to vary within each block of repeated columns, but their block sums are integers. Whether those sums generate every relation among the original columns is therefore controlled by the integer vectors in the original kernel.

\begin{theoremalpha}[Replication criterion; Corollary~\ref{cor:rational}]
For a real matrix $B$, some positive column replication has its kernel spanned by sign vectors if and only if $\ker B$ is rational. When this holds, every replication with sufficiently large even multiplicities has a sign-spanned kernel.
\end{theoremalpha}

For an arbitrary real kernel, we determine the eventual deficit as the dimension not spanned by its integer vectors. The main tool is an explicit multiplicity formula: it separates variations within repeated-column blocks from the span of their feasible sums. With no repeated columns, it retains the original span problem; its computational advantage comes from the repetition.

Two extremal consequences complement the realization theorem. Having $n-1$ orthogonal optimal signs forces a uniformly weighted complete graph at a Hadamard order. Without orthogonality, nonuniform product weights can attain maximal cut-span, but their positive directions have primitive integer representatives bounded by maximal binary determinants. This yields a complete finite classification through order eight and clarifies how orthogonality differs from spanning.

Our structural results build on several established descriptions. Spectral bounds, exactness, and recognition have been studied in \cite{DP93_maxcut,DP93_complexity,HLW21,MW2X,Bhardwaj2X,BhardwajHardness}. The facial theory of the elliptope and spectrahedra, including the Schur-independence characterization of simplicial cut faces, provides the geometric framework \cite{LP95,LP96,RamanaGoldman95,KuengTropp21}.

The multiplicity formula is an explicit instance of block-permutation symmetry and orbit averaging \cite{Margot10,HerrRehnSchurmann13,Bulutoglu23}. Its binary version concerns equality-knapsack polytopes. Lee records the possible range of their dimensions and proves full dimensionality for a special class \cite{Lee97}; the binary identity recovers that result. Laurent and Poljak's root-pattern aggregation groups equal coefficients when studying gap inequalities \cite[Section~2 and Theorem~4.4]{LP96Gap}. At gap zero, their scalar count equation agrees with ours, but their criterion concerns cut incidence vectors in edge space, whereas we compute the span of vertex sign vectors. Proposition~\ref{prop:binarydimension} records the general binary dimension identity, and the appendix verifies its specialization to Lee's result.

\subsection{Notation and organisation}
We write $\mathbb S^n$ for real symmetric matrices, with $\langle A,B\rangle=\operatorname{tr}(AB)$, and $\diag$ and $\Diag$ for extracting and forming a diagonal. We write $\one$ for the all-ones vector and $I$ for the identity matrix, with dimensions understood from context. The symbols $\spn$, $\conv$, and $\range$ denote real linear span, convex hull, and column space; $\spn\varnothing=\{0\}$, and dimensions of convex sets are affine dimensions. The product $u\circ v$ is entrywise, while $\otimes$ and $\oplus$ denote Kronecker product and direct sum. We use $\|u\|_\infty=\max_i|u_i|$ and $\relint$ for relative interior, and write $A(G)$ for the adjacency matrix of a simple graph $G$. A symmetric $H$ is copositive if $u^\top Hu\ge0$ for every $u\ge0$.

Unless stated otherwise, $W$ is nonnegative, symmetric, and zero-diagonal, with $n\ge2$ and no isolated vertices in its support graph. We write $\cut_W(x)$ for the cut weight and $W_i(x)$ for the cut weight incident to vertex $i$, with $W(x)=(W_i(x))_i$. Both $x$ and $-x$ are retained as sign vectors. For an exact instance, define
\begin{equation}\label{def:invariants}
\mathcal C(W)=\{x\in\sgn^n:\cut_W(x)=\MC(W)\},\quad
c(W)=\dim\spn\mathcal C(W),\quad d(W)=\rank Z^*.
\end{equation}
The optimal face is $F^*(W)$ and its cut hull is $P^*(W)=\conv\{xx^\top:x\in\mathcal C(W)\}$. Sign-spanning means $c(W)=n-d(W)$. A primal--dual optimal pair is strictly complementary if $\rank X+\rank Z^*=n$. The effective support is $\Keff=\spn\{\range X:X\in F^*(W)\}$. A sign family $s_1,\ldots,s_q$ is Schur independent if $\one$ and the products $s_j\circ s_k$ for $j<k$ are linearly independent.

For a matrix $Y$ with $n$ columns, put $\sigma(Y)=\dim\spn(\ker Y\cap\sgn^n)$. If $B$ has $r$ columns and $m\in\mathbb Z_{>0}^r$, then $B^{(m)}$ repeats column $t$ exactly $m_t$ times.

Complexity statements count bit operations on explicitly listed binary rational input; feasibility means existence of a real solution. We write $\mathsf{CH}$ for cut-hull recognition under the exactness promise and $\mathsf{SDF}$ for exact rational semidefinite feasibility. The symbol $\le_m^p$ denotes polynomial-time many-one reduction. Specialized notation is introduced with the arguments that use it.

Section~\ref{sec:prelim} recalls the certificates and proves their exact polynomial-time recovery. Section~\ref{sec:dimension} develops multiplicity, replication, and realization, answering the vector-spanning question. Section~\ref{sec:faces} then explains why spanning need not imply generation of the optimal face. This leads to deterministic recognition under strict complementarity in Section~\ref{sec:recognition}, the unique-cut embedding in Section~\ref{sec:arithmetic}, and the unweighted exact-feasibility equivalence in Section~\ref{sec:sdp-complexity}. Section~\ref{sec:orthogonal} returns to extremal cut spans and orthogonality. The appendix contains supporting constructions, certificate encodings, and finite classifications. Section~\ref{sec:discussion} concludes with further questions.

\section{Certificates and cut spans}
\label{sec:prelim}
We recall the standard certificate and complementary-slackness identities \cite{DP93_maxcut,HLW21,LP95,LP96}. Their role is to translate between a discrete cut and the kernel of a positive semidefinite matrix. In our normalization,
\begin{align}
\SDP(W)&=\max\{\tfrac14\langle L_W,X\rangle:X\in\Ell\},\label{eq:sdp}\\
\operatorname{D}(W)&=\min\{\one^\top s:Z=\Diag(s)-\tfrac14L_W\succeq0\}.\label{eq:dual}
\end{align}
The primal is strictly feasible at $I$ and has compact feasible region; the dual is strictly feasible for sufficiently large $s$. Both optima are attained and equal.

\begin{proposition}\label{prop:certificate}
For a sign vector $x$, let $Z_x=\Diag(W(x)/2)-L_W/4$. Then $Z_xx=0$, and $Z_x\succeq0$ if and only if the relaxation is exact and $x$ is a maximum cut. In an exact instance the optimal dual solution is unique and satisfies $s_i^*=W_i(x)/2$ for every maximum cut $x$. In particular, all maximum cuts have the same local cut profile, and $W_i(x)>\delta_i/2$ at every vertex.
\end{proposition}
\begin{proof}
The identity $(L_Wx)_i=2x_iW_i(x)$ gives $Z_xx=0$. If $Z_x\succeq0$, it is a feasible slack complementary to $xx^\top$, proving exactness and optimality. Conversely, complementary slackness for any optimal slack and cut gives $Z^*x=0$. Its $i$th coordinate forces $s_i^*=W_i(x)/2$, proving uniqueness and the common profile. Finally, a zero diagonal entry of a positive semidefinite matrix forces a zero row. Since $Z^*_{ij}=W_{ij}/4$ off the diagonal and there are no isolated vertices, $Z^*_{ii}>0$. Thus $W_i(x)/2-\delta_i/4>0$.
\end{proof}

Dual uniqueness holds more generally without exactness \cite[Theorem~2]{DP93_maxcut}\cite[Corollary~2]{deCarliSilvaTuncel19}; the proof above identifies the local profile in the exact case. Once one cut certifies exactness, this same slack describes all optimal cuts and the entire SDP face.

The certificate identity has an algorithmic consequence. On rational exact instances, the optimal diagonal lies on a known rational grid. The next bound turns a small objective gap into coordinatewise accuracy, making exact recovery possible even without strict complementarity.

\begin{proposition}[Recovering the optimal slack]\label{prop:slackrecovery}
Let $W$ be rational and nonnegative, let $D\ge1$ be an integer such that $DW$ is integral, and put $M=\sum_{i<j}W_{ij}$. Suppose the relaxation is exact. For every dual feasible $s$ satisfying $\one^\top s\le\SDP(W)+\varepsilon$, with $\varepsilon\ge0$, we have
\begin{equation}\label{eq:slackerror}
(s_i-s_i^*)^2\le Z(s)_{ii}\varepsilon
\le (M/2+\varepsilon)\varepsilon\qquad(1\le i\le n).
\end{equation}
If $0<\varepsilon\le[32D^2(M+2)]^{-1}$, rounding each $s_i$ to the nearest element of $(2D)^{-1}\mathbb Z$ recovers $s_i^*$. Consequently, the optimal dual slack can be computed from $W$ in deterministic polynomial time under the promise of exactness. Neither an optimal cut nor strict complementarity is required as input.
\end{proposition}
\begin{proof}
Fix an optimal cut $x$ for the analysis. Proposition~\ref{prop:certificate} gives $Z(s^*)x=0$ and $s_i^*=W_i(x)/2$. Therefore
\[
Z(s)x=\Diag(s-s^*)x,\qquad
x^\top Z(s)x=\one^\top(s-s^*)\le\varepsilon.
\]
Cauchy--Schwarz for the positive semidefinite form defined by $Z(s)$ gives
\[
(s_i-s_i^*)^2=(e_i^\top Z(s)x)^2
\le Z(s)_{ii}\,x^\top Z(s)x\le Z(s)_{ii}\varepsilon.
\]
Also $Z(s)_{ii}\le\operatorname{tr}Z(s)=\one^\top s-M/2\le M/2+\varepsilon$, since $\MC(W)\le M$. This proves \eqref{eq:slackerror}. Taking $\varepsilon=0$ also recovers uniqueness in the exact case.

Every $s_i^*$ belongs to $(2D)^{-1}\mathbb Z$. For the stated positive accuracy, $\varepsilon<1$ and
\[
(M/2+\varepsilon)\varepsilon
<\frac{M/2+1}{32D^2(M+2)}=\frac1{64D^2}.
\]
Thus $|s_i-s_i^*|<1/(8D)$, less than half the grid spacing, and rounding is unambiguous.

It remains to obtain a sufficiently accurate \emph{feasible} dual point in the bit model. Put $\Delta=\max_i\delta_i$, $c=\Delta/2+1$, and $U_0=\Delta+2$. Consider
\[
K=\{s:Z(s)\succeq0,\ s_i\le U_0\ (1\le i\le n)\}.
\]
It contains $s^*$ because $0\le s_i^*=W_i(x)/2\le\Delta/2$. Moreover, $\lambda_{\max}(L_W)\le2\Delta$, so $Z(c\one)\succeq I$. The Euclidean ball of radius $1/2$ about $c\one$ lies in $K$, and $K\subseteq[0,U_0]^n$ because $Z(s)_{ii}\ge0$ implies $s_i\ge\delta_i/4$.

To apply a standard bit-complexity theorem directly, represent $K$ by the affine PSD slice
\[
T(s)=Z(s)\oplus\Diag(U_0\one-s)\succeq0.
\]
The off-diagonal and off-block entries are fixed, and corresponding diagonal entries in the two blocks sum to $U_0-\delta_i/4$. These rational affine equations describe exactly the matrices $T(s)$. In this affine hull,
$\|T(s)-T(t)\|_F=\sqrt2\|s-t\|_2$.
Consequently a relative Frobenius ball of radius $1/2$ about the rational matrix $T(c\one)$ is feasible, and the whole slice lies in the ball of radius $2nU_0$ about that matrix. The objective $\one^\top s$ is the trace of the first block plus the constant $M/2$.

The bounded-SDP approximation theorem, in the precise rational form of \cite[Theorem~1.1]{deKlerkVallentin16}, now returns a rational feasible $s$ of value at most $\SDP(W)+\varepsilon$ in time polynomial in the input length and $\log(1/\varepsilon)$. The theorem is the classical ellipsoid guarantee, for which that reference also supplies an interior-point proof. All the required radii and the strictly feasible point have been exhibited above. Taking $D$ to be the product of the input denominators gives polynomial bit length, as do $M$, $\Delta$, and the chosen $\varepsilon$. Rounding and forming $Z(s^*)$ complete the algorithm. The cut $x$ was used only in the proof of the error bound.
\end{proof}

This result separates certificate recovery from support recovery. The dual kernel can be computed exactly from the weights, but it may still contain directions absent from every primal optimum. That latter issue is the source of the recognition difficulty in Section~\ref{sec:sdp-complexity}.

\begin{proposition}\label{prop:face}
For an exact instance,
\begin{align}
\mathcal C(W)&=\ker Z^*\cap\sgn^n,\label{eq:cuts}\\
F^*(W)&=\{X\in\Ell:Z^*X=0\},\label{eq:face}\\
1\le c(W)&\le n-d(W)\le n-1.\label{eq:rankbound}
\end{align}
Sign-spanning implies the existence of a strictly complementary optimal pair.
\end{proposition}
\begin{proof}
For feasible $X$, the duality gap is $\langle Z^*,X\rangle$. Positivity makes it zero precisely when $Z^*X=0$, giving the first two identities and the kernel bound. Some edge has positive weight, so $Z^*\ne0$. If $x^{(1)},\ldots,x^{(q)}$ are optimal signs forming a basis of $\ker Z^*$, their averaged lifts $q^{-1}\sum_jx^{(j)}(x^{(j)})^\top$ form an optimal matrix with range $\ker Z^*$, proving strict complementarity.
\end{proof}

The converse fails: a primal optimum may use every direction in the dual kernel even when the optimal signs do not span it. Moreover, the affine dimension of the face concerns variation among matrices, rather than the span of vectors. Section~\ref{sec:faces} gives small examples separating all three notions.

\section{The dimension formula, replication, and realization}
\label{sec:dimension}

The certificate description from Section~\ref{sec:prelim} turns the study of optimal cuts into a question about sign vectors in a linear subspace. Factoring the dual slack makes that question particularly concrete: if $Z=Y^\top Y$, then an optimal cut is a choice of signs for the columns of $Y$ whose signed sum is zero. We first determine how repeated columns contribute to the span of these sign vectors. We then introduce repetitions deliberately, obtaining the arithmetic replication criterion and the realization theorem.

The following reformulation of the certificate will be used in both directions. Starting from an exact instance, it produces a Gram factor whose kernel sign vectors are the optimal cuts. Starting from a suitable factor, it produces an exact Max-Cut instance.

\begin{proposition}\label{prop:dnn}
Let $W\ge0$ have no isolated vertices, and let $x^{(1)},\ldots,x^{(k)}\in\sgn^n$, where $k\ge1$. The relaxation is exact with each $x^{(j)}$ optimal if and only if there is a doubly nonnegative matrix $Z$ such that
\[
W_{ij}=4Z_{ij}\quad(i\ne j),
\qquad
Zx^{(j)}=0\quad(1\le j\le k).
\]
In that case $Z=Z^*$. If $Z=Y^\top Y$ and $y_i$ denotes the $i$th column of $Y$, then
\[
W_{ij}=4\langle y_i,y_j\rangle\quad(i\ne j),
\qquad
\langle y_i,y_j\rangle\ge0,
\qquad
Yx^{(j)}=\sum_i x_i^{(j)}y_i=0.
\]
The factor $Y$ may be chosen with $d(W)$ independent rows, and every column is nonzero.
\end{proposition}

\begin{proof}
For an exact instance, the optimal slack is positive semidefinite, its off-diagonal entries are $W_{ij}/4\ge0$, and its diagonal entries are nonnegative. It is therefore doubly nonnegative, and complementary slackness gives $Z^*x^{(j)}=0$ for every optimal cut.

Conversely, suppose that $Z$ satisfies the stated conditions. Set $s_i=Z_{ii}+\delta_i/4$. Then
\[
Z=\Diag(s)-\tfrac14L_W
\]
is dual feasible and complementary to each primal feasible matrix $x^{(j)}(x^{(j)})^\top$. Thus the relaxation is exact and all the prescribed cuts are optimal. Proposition~\ref{prop:certificate} identifies $Z$ with the unique optimal slack.

A Gram factorization with $\rank Z=d(W)$ independent rows gives the remaining identities, using $\ker(Y^\top Y)=\ker Y$. A zero column would make every edge weight incident to the corresponding vertex zero, contrary to the assumption that there are no isolated vertices.
\end{proof}

For such a factor, $c(W)=\sigma(Y)$. The problem is therefore to determine how much of $\ker Y$ is generated by its sign vectors. Repeated columns provide a useful separation. If the same column occurs several times, its contribution to $Yx$ depends only on the number of negative signs assigned to those positions. Once those numbers are fixed, signs can still be rearranged within each group.

These rearrangements produce kernel directions whenever a group contains both signs in at least one feasible assignment. We call such a group active. The next theorem shows that each active group contributes its entire zero-sum subspace. The remaining dimension is determined by the feasible sums across groups.

\begin{theorem}[Multiplicity formula]\label{thm:multiplicity}
Let $Y\in\mathbb R^{d\times n}$. Let $b_1,\ldots,b_r$ be its distinct columns, with multiplicities $m_1,\ldots,m_r$, and put $m=(m_1,\ldots,m_r)$. Define
\begin{align*}
T&=\sum_{t=1}^r m_tb_t,\\
S&=\left\{s\in\mathbb Z^r:0\le s_t\le m_t,\ \sum_{t=1}^r s_tb_t=T/2\right\},\\
\mathcal A&=\{t:\text{there exists }s\in S\text{ with }0<s_t<m_t\},\\
Q&=\spn\{m-2s:s\in S\}.
\end{align*}
Then
\begin{equation}\label{eq:dimensionformula}
\sigma(Y)=\sum_{t\in\mathcal A}(m_t-1)+\dim Q.
\end{equation}
In particular, the sign vectors span $\ker Y$ if and only if
\begin{enumerate}[label=(\roman*)]
\item every $t$ with $m_t\ge2$ belongs to $\mathcal A$; and \label{cond:1}
\item $Q=\{z\in\mathbb R^r:\sum_tz_tb_t=0\}$. \label{cond:2}
\end{enumerate}
\end{theorem}

\begin{proof}
Let $V_t$ index the columns equal to $b_t$, and let $\pi:\mathbb R^n\to\mathbb R^r$ record the sums over these blocks:
\[
\pi(w)_t=\sum_{i\in V_t}w_i.
\]
For a sign vector $x$, let $s_t(x)$ count its negative entries in $V_t$. Then
\[
\pi(x)=m-2s(x),
\qquad
Yx=T-2\sum_t s_t(x)b_t.
\]
Thus $x\in\ker Y$ precisely when $s(x)\in S$, and every count vector in $S$ is realized by a sign vector. In particular, if
\[
U=\spn(\ker Y\cap\sgn^n),
\]
then $\pi(U)=Q$.

For each block $V_t$, let $K_t$ be the subspace of vectors supported on $V_t$ whose coordinates sum to zero. Suppose first that $t\in\mathcal A$. Some kernel sign vector has both signs in this block. By permuting its entries within the block, we may place opposite signs at any prescribed pair of distinct positions $i,j\in V_t$. Interchanging those signs gives another kernel sign vector, and their difference is $\pm2(e_i-e_j)$. These differences span $K_t$, so $K_t\subseteq U$. If $t\notin\mathcal A$, every kernel sign vector is constant on $V_t$, and hence every vector in $U$ is constant there.

To separate these within-block directions from the block sums, average over all permutations within each block. This gives a linear projection $P$ onto the block-constant vectors, with
\[
(Pw)_i=\frac{\pi(w)_t}{m_t}\qquad(i\in V_t).
\]
The set of kernel sign vectors is invariant under these permutations, so $P(U)\subseteq U$. For every kernel sign vector $x$, the difference $x-Px$ belongs to $\bigoplus_{t\in\mathcal A}K_t$. Consequently,
\[
U=\left(\bigoplus_{t\in\mathcal A}K_t\right)\oplus P(U).
\]
The sum is direct because a block-constant vector with zero sum on every block is zero. Moreover, $\pi$ is injective on the block-constant subspace and
\[
\pi(P(U))=\pi(U)=Q.
\]
Taking dimensions proves \eqref{eq:dimensionformula}.

Finally, the distinct columns span the column space of $Y$, so the space in condition~\ref{cond:2} has dimension $r-\rank Y$. The two contributions in \eqref{eq:dimensionformula} satisfy
\[
\sum_{t\in\mathcal A}(m_t-1)\le n-r,
\qquad
\dim Q\le r-\rank Y.
\]
The first bound is attained precisely when every block of size at least two is active, and the second precisely when condition~\ref{cond:2} holds. Their sum equals $\dim\ker Y=n-\rank Y$ exactly under the two stated conditions. Neither nonzero columns nor full row rank is required.
\end{proof}

The two conditions address different possible losses of dimension. An inactive repeated-column block loses its within-block zero-sum directions. Even when every such block is active, the feasible block sums may fail to span all relations among the distinct columns. The following example isolates this second obstruction.

\begin{example}
For $Y=(1,1,2,2)$, the two blocks have multiplicities $(2,2)$ and the only feasible negative-sign count is $(1,1)$. Both blocks are active, but their feasible block sums are zero. Hence $\sigma(Y)=1+1=2$, whereas $\dim\ker Y=3$. The kernel signs satisfy $x_1=-x_2$ and $x_3=-x_4$: all variation occurs within the blocks. Activity requires an interior count in some feasible pattern, not variation of that count across patterns.
\end{example}

The same block decomposition applies to arbitrary block-permutation-invariant binary sets. Under $x\mapsto(\one-x)/2$, it identifies $\sigma(Y)$ with the affine dimension of the corresponding nonempty equality-constrained binary polytope. Appendix~\ref{ec:binary} records the general identity, and Appendix~\ref{ec:lee} recovers Lee's full-dimensionality result. For a rational input of bit length $L$ with $r$ distinct columns, direct enumeration gives time $(n+1)^r\operatorname{poly}(L)$; this is polynomial for fixed $r$, not a general polynomial-time algorithm.

To return from these span calculations to Max-Cut, the factor must have nonnegative pairwise inner products. Under that condition, any feasible sign pattern supplies an exactness certificate, and the multiplicity formula computes the span of all optimal cuts.

\begin{corollary}\label{cor:construction}
Let $Y$ have nonzero columns with pairwise nonnegative inner products, and define
\[
W_{ij}=4\langle y_i,y_j\rangle\quad(i\ne j),
\qquad W_{ii}=0.
\]
Suppose that the support graph has no isolated vertices. If the feasible count set $S$ is nonempty, then the relaxation is exact,
\[
Z^*=Y^\top Y,
\qquad
c(W)=\sigma(Y).
\]
Conversely, every nonnegatively weighted exact instance without isolated vertices has such a representation with $S\ne\varnothing$.
\end{corollary}

\begin{proof}
A count vector in $S$ gives a sign vector in $\ker Y$, so Proposition~\ref{prop:dnn} applies. Equation~\eqref{eq:cuts} identifies all optimal cut sign vectors with the sign vectors in $\ker Y$, proving $c(W)=\sigma(Y)$. Conversely, factor the optimal slack of an exact instance and apply Proposition~\ref{prop:dnn}.
\end{proof}

\subsection{Replication and realization of dual certificates}
\label{sec:replication}

The multiplicity formula suggests a way to construct matrices whose kernels are spanned by sign vectors. Repeating every column an even number of times allows a balanced choice of signs within every block. This immediately supplies kernel sign vectors and makes all within-block zero-sum directions available. The remaining question is whether the feasible block sums generate every relation among the original columns.

Because these block sums are integers, the answer depends on the integer vectors in the original kernel. We first make this dependence exact. Throughout the replication results, the columns of $B$ need not be distinct or nonzero, and $B$ need not have full row rank.

For $k\in\mathbb Z_{>0}^r$, let
\[
L_k(B)=\spn\{z\in\ker B\cap\mathbb Z^r:|z_t|\le k_t\text{ for every }t\}.
\]
This is the part of the original kernel generated by integer vectors within the replication bounds.

\begin{theorem}[Replication]\label{thm:replication}
Let $B\in\mathbb R^{d\times r}$ and $k\in\mathbb Z_{>0}^r$. Set $Y=B^{(2k)}$ and $n=2\sum_{t=1}^r k_t$. Then
\begin{equation}\label{eq:replication}
\sigma(Y)=\sum_{t=1}^r(2k_t-1)+\dim L_k(B).
\end{equation}
In particular, the sign vectors in $\ker Y$ span $\ker Y$ if and only if $L_k(B)=\ker B$.
\end{theorem}

\begin{proof}
Let
\[
V=\spn(\ker Y\cap\sgn^n),
\]
and let $C:\mathbb R^n\to\mathbb R^r$ record the sums over the $r$ replication blocks. Then $Y=BC$, and
\[
\ker C
=
\bigoplus_{t=1}^r
\{u\in\mathbb R^{2k_t}:\one^\top u=0\},
\qquad
\dim\ker C=n-r.
\]

Every sign vector balanced within each block belongs to $\ker Y$. For any two positions in the same block, choose such a vector with opposite signs at those positions and interchange them. The difference of the resulting kernel sign vectors is twice the corresponding coordinate difference. These differences span $\ker C$, so $\ker C\subseteq V$.

The block sums of a sign vector are precisely the vectors $2z$ with $z\in\mathbb Z^r$ and $|z_t|\le k_t$ for every $t$. The additional condition $Yx=0$ is equivalent to $Bz=0$. Hence
\[
C(\ker Y\cap\sgn^n)
=
\{2z:z\in\ker B\cap\mathbb Z^r,\ |z_t|\le k_t\text{ for every }t\}.
\]
Taking spans removes the common factor two, giving $C(V)=L_k(B)$. Since $\ker C\subseteq V$, the kernel of the restriction $C|_V$ is exactly $\ker C$. Rank--nullity applied to $C|_V$ now yields
\[
\sigma(Y)=\dim V
=
\dim\ker C+\dim L_k(B)
=
n-r+\dim L_k(B),
\]
which proves \eqref{eq:replication}.

Finally, $\rank Y=\rank B$, so $\dim\ker Y=n-\rank B$. Since $L_k(B)\subseteq\ker B$, the equality $\sigma(Y)=\dim\ker Y$ holds precisely when $L_k(B)=\ker B$.
\end{proof}

The first term in \eqref{eq:replication} is already as large as possible. Increasing the replication sizes affects the remaining obstruction by admitting more integer vectors into the coordinate box defining $L_k(B)$. A rational kernel has a basis of integer vectors, so a sufficiently large box contains enough vectors to span it. Conversely, the block sums of sign vectors can never generate a subspace that has no integer spanning set. This gives the exact criterion.

\begin{corollary}[Rationality is the exact obstruction]\label{cor:rational}
For $B\in\mathbb R^{d\times r}$, the following are equivalent:
\begin{enumerate}[label=(\alph*)]
\item Some positive replication $B^{(m)}$, with $m\in\mathbb Z_{>0}^r$ of arbitrary parity, has a sign-spanned kernel.
\item The subspace $\ker B$ is rational.
\item There is $k^0\in\mathbb Z_{>0}^r$ such that every $B^{(2k)}$ with $k\ge k^0$ coordinatewise has a sign-spanned kernel.
\end{enumerate}
In particular, sufficiently large uniform even replication suffices whenever any positive replication can suffice.
\end{corollary}

\begin{proof}
For a positive replication, let $C$ be the block-sum map, so that $B^{(m)}=BC$. Every block is nonempty, hence $C$ is onto and
\[
C(\ker B^{(m)})=\ker B.
\]
Indeed, any $z\in\ker B$ can be lifted to a vector whose block sums are $z_t$, and every such lift belongs to $\ker(BC)$. If $\ker B^{(m)}$ is spanned by sign vectors, their images under $C$ are integer vectors spanning $\ker B$. Thus (a) implies (b), without a parity assumption.

Suppose now that $\ker B$ is rational. Choose a real vector-space basis consisting of integer vectors $h^{(1)},\ldots,h^{(\ell)}$, where $\ell=\dim\ker B$, and set
\[
k_t^0
=
\max\bigl(\{1\}\cup\{|h_t^{(j)}|:1\le j\le\ell\}\bigr).
\]
For $k\ge k^0$, every basis vector belongs to the defining set of $L_k(B)$, so $L_k(B)=\ker B$. Theorem~\ref{thm:replication} proves (c). The same argument includes the zero kernel by taking an empty basis. Finally, (c) implies (a).
\end{proof}

Even when the kernel is not rational, the replication formula determines exactly how much of it remains inaccessible. The integer kernel vectors span a fixed subspace, and a finite collection of them suffices to span that subspace. Once the replication box contains such a collection, the deficit stabilizes.

\begin{corollary}\label{cor:deficit}
For a real matrix $B\in\mathbb R^{d\times r}$, put $\ell_{\mathbb Z}(B)=\dim\spn(\ker B\cap\mathbb Z^r)$. There exists $k^0\in\mathbb Z_{>0}^r$ such that, for every $k\ge k^0$,
\[
\dim\ker B^{(2k)}-\sigma(B^{(2k)})
=
\dim\ker B-\ell_{\mathbb Z}(B).
\]
The same conclusion holds for all sufficiently large uniform even replications.
\end{corollary}

\begin{proof}
Choose a basis of $\spn(\ker B\cap\mathbb Z^r)$ from the integer vectors generating it, and choose $k^0$ large enough to contain this finite basis in its coordinate box. For $k\ge k^0$,
\[
L_k(B)=\spn(\ker B\cap\mathbb Z^r).
\]
With $n=2\sum_tk_t$, Theorem~\ref{thm:replication} gives
\[
\dim\ker B^{(2k)}-\sigma(B^{(2k)})
=
(n-\rank B)-(n-r+\ell_{\mathbb Z}(B))
=
\dim\ker B-\ell_{\mathbb Z}(B).
\]
If $\ell_{\mathbb Z}(B)=0$, take the empty basis and $k^0=\one$.
\end{proof}

When $\ker B$ is rational, the uniform threshold is $\kappa(B)=\min\{k\in\mathbb Z_{>0}:L_{k\one}(B)=\ker B\}$ and depends only on $\ker B$. In the nontrivial case $0<\dim\ker B<r$, the maximal minors of a full-row-rank integer matrix defining that kernel yield a real basis consisting of integer vectors and bound $\kappa(B)$ by their largest absolute value. The zero and full kernels have threshold one, as does a totally unimodular factor. Proposition~\ref{prop:replicationbound} gives the precise bound and its proof.

We now apply replication to a prescribed doubly nonnegative matrix. Its Gram factor supplies the columns to be repeated. Positive diagonal entries ensure that copies of each column are joined by positive-weight edges, while balanced signs within each replication block certify exactness. The arithmetic criterion then determines when the optimal cuts span the entire dual kernel.

\begin{theorem}[Realization]\label{thm:realization}
Let $Z\in\mathbb R^{r\times r}$ be doubly nonnegative with positive diagonal, rank $d$, and rational kernel. For $k\in\mathbb Z_{>0}$, let $n=2kr$ and define
\[
W_{ij}=4(Z\otimes J_{2k})_{ij}\quad(i\ne j),
\qquad
W_{ii}=0.
\]
For every $k\ge1$, the support graph has no isolated vertices, the relaxation is exact, and
\[
Z^*=Z\otimes J_{2k},
\qquad
d(W)=d.
\]
For every $k\ge\kappa(Z)$, moreover,
\[
c(W)=n-d,
\]
so the optimal cut sign vectors span the entire dual kernel. The threshold $\kappa(Z)$ depends only on $\ker Z$. If the off-diagonal support of $Z$ is connected, then the support graph of $W$ is connected for every $k\ge1$.

The matrix $Z\otimes J_{2k}$ is completely positive if and only if $Z$ is. When these matrices are completely positive, their cp-ranks are equal.
\end{theorem}

\begin{proof}
Choose a Gram factor $Z=B^\top B$ with $B\in\mathbb R^{d\times r}$, and set $Y=B^{(2k\one)}$. Then
\[
Y^\top Y=Z\otimes J_{2k},
\qquad
\ker B=\ker Z.
\]
The columns of $B$ are nonzero because $Z$ has positive diagonal, and their pairwise inner products are nonnegative because $Z$ is entrywise nonnegative. Within the replication block corresponding to $b_t$, every edge has weight $4Z_{tt}>0$. Each block has at least two vertices, so the support graph has no isolated vertices. Each nonzero off-diagonal entry of $Z$ joins all vertices in the corresponding pair of blocks; hence connected support of $Z$ gives connected support of $W$.

For every $k\ge1$, choosing equally many positive and negative signs within each block gives a sign vector in $\ker Y$. Corollary~\ref{cor:construction} therefore proves exactness and gives
\[
Z^*=Y^\top Y,
\qquad
c(W)=\sigma(Y).
\]
Also,
\[
d(W)=\rank(Z\otimes J_{2k})=\rank Z=d
\]
for every $k\ge1$. Since $\ker B=\ker Z$ is rational, $\kappa(B)=\kappa(Z)$ is finite. For every $k\ge\kappa(Z)$, the definition of this threshold and Theorem~\ref{thm:replication} give
\[
c(W)=\sigma(Y)=n-\rank B=n-d.
\]

If $Z=N^\top N$ with $N\ge0$, repeating each column of $N$ exactly $2k$ times gives a nonnegative factorization of $Z\otimes J_{2k}$ with the same number of rows. Conversely, selecting one index from each replication block recovers $Z$ as a principal submatrix. Restricting any nonnegative factor to those indices gives a nonnegative factor of $Z$ with no additional rows. These two operations prove both the equivalence of complete positivity and equality of cp-ranks.
\end{proof}

The construction differs from the vertex splitting in \cite[Theorem~9 and Lemma~10]{Bhardwaj2X}. In that operation, copies of a vertex are mutually nonadjacent, original edge weights are distributed among pairs of copies, and primal optima lift with unchanged rank. Here copies of index $t$ form a clique with edge weight $4Z_{tt}$, while edges between copies of distinct indices $t,u$ have weight $4Z_{tu}$. The preserved rank is the dual rank, and the replication criterion controls whether optimal cut vectors span the resulting kernel.

Rational matrix entries suffice for the realization theorem, but only rationality of the kernel is needed. Exactness itself requires neither: for $B=(1,\sqrt2)$, uniform even replication has $n=4k$ columns and always gives an exact instance, while $\sigma(B^{(2k\one)})=n-2<n-1=\dim\ker B^{(2k\one)}$. The missing direction reflects the absence of nonzero integer vectors in $\ker B$. Appendix~\ref{ec:irrational} gives the full calculation.

The realization preserves obstructions to nonnegative Gram factors. Applying it to a suitable family of matrices gives connected non-completely-positive certificates at every rank at least three.

\begin{corollary}\label{cor:nonormalform}
For every $d\ge3$, there exists a connected nonnegatively weighted exact Max-Cut instance whose optimal dual slack has rank $d$, is not completely positive, and has a kernel spanned by optimal cut sign vectors.
\end{corollary}

The rank-three matrix in Example~\ref{ex:horn14} provides the initial obstruction. For larger ranks, adjoining an identity block raises the rank; a rational nonnegative rank-one update joins the support without changing the kernel or destroying the negative copositive pairing. Applying Theorem~\ref{thm:realization} then gives the result. Appendix~\ref{ec:connected} supplies the update and verifies these properties.

The following example makes the rank-three obstruction explicit. It also shows that allowing unequal replication sizes can substantially reduce the number of vertices needed for full spanning.

\begin{example}[A fourteen-vertex non-completely-positive certificate]\label{ex:horn14}
Let
\[
B=\begin{pmatrix}1&0&-1&-1&0\\0&1&1&-1&-1\\1&1&1&1&1\end{pmatrix},
\qquad
Z_0=B^\top B=\begin{pmatrix}2&1&0&0&1\\1&2&2&0&0\\0&2&3&1&0\\0&0&1&3&2\\1&0&0&2&2\end{pmatrix}.
\]
This rational doubly nonnegative matrix has rank three and cycle support $C_5$. The classical copositive Horn matrix $H=J_5-2A(C_5)$ \cite{berman2003completely} gives
\[
\langle H,Z_0\rangle=12-2(1+2+1+2+1)=-2,
\]
so $Z_0$ is not completely positive. Indeed, every nonnegative Gram factor would express this pairing as a sum of nonnegative quadratic forms.

Repeat the columns of $B$ with multiplicities $m=(4,4,3,1,2)$, and put $Y=B^{(m)}$ and $W_{ij}=4\langle y_i,y_j\rangle$ off the diagonal. The support is connected. Every integer vector $c\in\ker B$ has the form
\[
(-2v-2t,\ 4v+3t,\ -3v-2t,\ v,\ t),\qquad v,t\in\mathbb Z.
\]
The bounds $|c_i|\le m_i$ and parity constraints $c_i\equiv m_i\pmod2$ leave exactly the block sums $\pm a,\pm b$, where
\[
a=(2,-2,1,1,-2),\qquad b=(2,-4,3,-1,0).
\]
These are independent; blocks $1,2,3,5$ are active and block $4$ has size one. Therefore
\[
c(W)=(4-1)+(4-1)+(3-1)+(2-1)+2=11=14-3.
\]
The instance is therefore exact and sign-spanned, with a non-completely-positive slack of rank three. It has $56$ optimal cuts up to sign and optimum $212$. The uniform even threshold is $\kappa(B)=2$, requiring twenty vertices, while fourteen is minimal among positive replications of this factor. Appendices~\ref{ec:fourteen} and~\ref{ec:minimum} verify the counts, threshold, and restricted minimum.
\end{example}

At dual ranks one and two, nonnegative Gram factors of minimum dimension are available independently of sign-spanning \cite{BrandtsKrizek16}. At rank three, complete positivity and equality of rank with cp-rank are separate requirements: the sign-spanned slack $I+\tfrac12A(C_4)$ has rank three and cp-rank four. Appendix~\ref{ec:lowrank} gives the rank-two normal form and verifies this example. The rank-one case will also be used in Section~\ref{sec:orthogonal}.

\begin{proposition}\label{thm:rankone}
An exact instance has dual rank one precisely when $W_{ij}=4v_iv_j$ for $i\ne j$, with $v>0$ and $v^\perp\cap\sgn^n\ne\varnothing$. In this case $Z^*=vv^\top$ and $c(W)=\sigma(v^\top)$.
\end{proposition}
\begin{proof}
A rank-one doubly nonnegative slack with positive diagonal is $vv^\top$ with $v>0$. Conversely, such an outer product and a kernel sign vector give a complementary certificate. Equation~\eqref{eq:cuts} identifies the optimal signs.
\end{proof}

\section{Optimal faces and cut hulls}
\label{sec:faces}\label{sec:hull}
Averaging a spanning family of optimal cut matrices gives an optimum of full rank on the dual kernel. Does every other optimum also arise as a mixture of cuts? The complete graphs $K_4$ and $K_6$ give opposite answers. We use them to explain the distinction, then give the Schur-independence test for cut generation. A further example shows why recognition must sometimes work in a proper subspace of the dual kernel.

Let the columns of $U\in\mathbb R^{n\times q}$ be a basis of $\ker Z^*$. Complementary slackness gives the coordinates
\begin{equation}\label{eq:URU}
F^*(W)=\{URU^\top:R\succeq0,\ \diag(URU^\top)=\one\}.
\end{equation}
For clarity, let $U^+=(U^\top U)^{-1}U^\top$. Any PSD matrix $X$ supported on $\range U$ has the unique coefficient matrix $R=U^+X(U^+)^\top\succeq0$, and $URU^\top=X$. This proves both the representation and injectivity. If a feasible $R\succ0$ exists, sufficiently small perturbations within the affine diagonal constraints remain positive definite. Positivity therefore adds no affine equalities, and
\begin{equation}\label{eq:generalfacedim}
\dim F^*(W)=\binom{q+1}{2}-\rank\bigl(R\mapsto\diag(URU^\top)\bigr).
\end{equation}
These are the spectrahedral face description and the Li--Tam dimension formula, specialized to the elliptope \cite{RamanaGoldman95,LiTam94}\cite[Section~31.5]{de1997geometry}. The positivity hypothesis is exactly strict complementarity. Without it, the feasible matrices can be confined to a proper subspace of the dual kernel.

\begin{example}\label{prop:completeface}\label{ex:k4triangle}\label{ex:notconvex}
For the unit-weight complete graph $K_n$ of even order, the optimal slack is $J_n/4$, balanced cuts attain $n^2/4$, and their sign vectors span $\one^\perp$. Thus
\[
F^*(K_n)=\{X\in\Ell:X\one=0\},\qquad c(K_n)=n-1.
\]
For even $n\ge4$, this face has dimension $n(n-3)/2$. The $n$ row-sum equations are independent in the off-diagonal coordinates: a linear dependence with coefficients $a_i$ would give $a_i+a_j=0$ for every $i\ne j$, which forces every $a_i=0$ when $n\ge3$. There are $\binom n2$ off-diagonal coordinates, and $\frac n{n-1}(I-J_n/n)$ is positive definite on $\one^\perp$, which proves the count.

For $K_4$, take the three balanced signs
\[
s_1=(1,1,-1,-1)^\top,\quad s_2=(1,-1,1,-1)^\top,\quad s_3=(1,-1,-1,1)^\top.
\]
Writing $U=[s_1\ s_2\ s_3]$, the unit-diagonal equations in \eqref{eq:URU} force $R$ to be diagonal with trace one. Hence $F^*(K_4)$ is the triangle generated by these three cut matrices: its dimension is two, although the cut-span dimension is three.

For $K_6$, let $a,b,c$ be unit vectors with $a+b+c=0$. The Gram matrix of $(a,a,b,b,c,c)$ has zero row sums and is optimal, but
\[
X_{13}+X_{15}+X_{35}=-\tfrac32<-1,
\]
violating an inequality satisfied by every cut matrix: for signs $x_1,x_3,x_5$, the sum $x_1x_3+x_1x_5+x_3x_5$ is either $3$ or $-1$. The same lower bound holds for convex combinations. Thus $K_6$ is sign-spanned but its optimal face strictly contains its cut hull. Figure~\ref{fig:cut-hull-contrast} illustrates the contrast: higher rank alone does not distinguish points inside the cut hull from those outside it.
\end{example}

\begin{figure}[htbp]
\centering
\begin{minipage}[t]{0.48\linewidth}
\centering
\textbf{(a) Optimal face of $K_4$}\par\medskip
\begin{tikzpicture}[every node/.style={font=\small},line join=round]
\coordinate (A) at (-1.65,0);
\coordinate (B) at (1.65,0);
\coordinate (C) at (0,2.75);
\coordinate (O) at (0,0.917);
\fill[cutblue!8] (A)--(B)--(C)--cycle;
\draw[cutblue,thick] (A)--(B)--(C)--cycle;
\foreach \p in {A,B,C} \fill[cutblue] (\p) circle (2pt);
\node[anchor=north east] at (A) {$X_1$};
\node[anchor=north west] at (B) {$X_2$};
\node[anchor=south] at (C) {$X_3$};
\fill[cutteal] (O) circle (2pt);
\node[above=3pt,align=center] at (O) {$\overline X$\\[-1pt]rank $3$};
\fill[cutorange] (0,0) circle (2pt);
\node[below=4pt,align=center] at (0,0) {$\tfrac12(X_1+X_2)$\\[-1pt]rank $2$};
\end{tikzpicture}\par\smallskip
{\small $X_j=s_js_j^\top$,\quad $\overline X=\tfrac13\sum_{j=1}^3X_j$.\par\smallskip
Vertices: rank $1$; open edges: rank $2$;\par
interior: rank $3$.}
\end{minipage}\hfill
\begin{minipage}[t]{0.48\linewidth}
\centering
\textbf{(b) Gram representation of a $K_6$ optimum}\par\medskip
\begin{tikzpicture}[every node/.style={font=\small},
  >={Stealth[length=2mm]},vec/.style={->,thick}]
\coordinate (O) at (0,0);
\coordinate (A) at (90:1.2);
\coordinate (B) at (210:1.2);
\coordinate (C) at (330:1.2);
\draw[gray!30] (O) circle (1.2);
\draw[vec,cutblue] (O)--(A);
\draw[vec,cutteal] (O)--(B);
\draw[vec,cutorange] (O)--(C);
\fill (O) circle (1.5pt);
\node[above,cutblue] at (A) {$v_1=v_2=a$};
\node[below left,cutteal] at (B) {$v_3=v_4=b$};
\node[below right,cutorange] at (C) {$v_5=v_6=c$};
\node[below] at (0,-1.3) {$a+b+c=0$};
\end{tikzpicture}\par\smallskip
{\small Each arrow represents two vertices.\par\smallskip
$X=(\langle v_i,v_j\rangle)_{i,j=1}^6$,\quad $\rank X=2$.\par\smallskip
$X_{13}+X_{15}+X_{35}:$\par
$\ge-1$ for every cut matrix;\quad $-\tfrac32$ here.}
\end{minipage}
\caption{Both instances are exact and sign-spanned. Left: a face in matrix space; every point is a convex combination of the three cut matrices. Right: a configuration of unit Gram vectors whose optimal matrix violates a cut inequality. The panels depict different spaces and show that sign-spanning does not determine whether the optimal face equals its cut hull.}
\label{fig:cut-hull-contrast}
\end{figure}

The $K_4$ calculation suggests the general criterion. With a sign basis, the off-diagonal entries of $R$ multiply entrywise products of distinct basis vectors. Independence of these products together with $\one$ forces every feasible $R$ to be diagonal. This is the classical Schur-independence criterion for simplicial cut faces \cite{LP96,KuengTropp21}.
\FloatBarrier

\begin{theorem}\label{thm:hull}
Suppose the instance is sign-spanned, and let $s_1,\ldots,s_q$ be an optimal sign basis of $\ker Z^*$, where $q=n-d(W)$. Put
\[
P=\spn\bigl(\{\one\}\cup\{s_j\circ s_k:j<k\}\bigr).
\]
Then
\begin{equation}\label{eq:facedim}
\dim F^*(W)=\binom{q+1}{2}-\dim P\ge q-1.
\end{equation}
The following are equivalent: the basis is Schur independent; the face has dimension $q-1$; and $F^*(W)=P^*(W)$. In this case the face is the simplex with vertices $s_js_j^\top$, and the optimal signs are exactly $\pm s_j$. In particular, $1+\binom q2\le n$.
\end{theorem}
\begin{proof}
Let $U=[s_1\ \cdots\ s_q]$ and $\Phi(R)=\diag(URU^\top)$. Since the columns are sign vectors,
\begin{equation}\label{eq:phi}
\Phi(R)=(\operatorname{tr}R)\one+2\sum_{j<k}R_{jk}(s_j\circ s_k).
\end{equation}
Thus $\range\Phi=P$ and $I/q$ is a positive definite feasible coefficient matrix. Formula~\eqref{eq:generalfacedim} proves \eqref{eq:facedim}; the lower bound follows from $\dim P\le1+\binom q2$. Equality is precisely Schur independence. Under that independence, $\Phi(R)=\one$ forces $R$ to be diagonal with trace one, giving the stated simplex. Its only rank-one points are its vertices because $U$ has full column rank.

Conversely, a face of the elliptope generated by cut matrices is a simplex \cite[Theorem~31.5.8]{de1997geometry}, and its vertex signs are Schur independent \cite[Theorem~3.4]{KuengTropp21}. They are therefore linearly independent and span the range of every matrix in the face. Since $I/q$ is feasible in the coefficient coordinates, this range span is $\ker Z^*$, so there are exactly $q$ vertices. The face has dimension $q-1$, completing the equivalence. The counting bound follows from $P\subseteq\mathbb R^n$.
\end{proof}

The dimension $\dim P$ is consequently independent of the chosen cut basis. Without sign-spanning, the same classical cut-face characterization still gives the necessary bound $1+\binom{c(W)}2\le n$ whenever $F^*(W)=P^*(W)$. The relevant subspace is then the effective support, rather than necessarily the full dual kernel; Section~\ref{sec:recognition} makes this precise.

\begin{example}\label{ex:supportdistinction}
For product weights $W_{ij}=4v_iv_j$ with a positive $v$ admitting a kernel sign vector, the slack is $vv^\top$ and the face is $F_v=\{X\in\Ell:Xv=0\}$. When $v=(1,1,2,2)^\top$, the signs satisfy $x_1=-x_2$ and $x_3=-x_4$, so $c(W)=2$. Nevertheless,
\[
X=\begin{pmatrix}1&0&-1/4&-1/4\\0&1&-1/4&-1/4\\-1/4&-1/4&1&-3/4\\-1/4&-1/4&-3/4&1\end{pmatrix}
\]
is optimal and has eigenvalues $0,1,5/4,7/4$. It has full rank on the three-dimensional dual kernel. Thus strict complementarity does not imply sign-spanning. The four equations $Xv=0$ are independent in the six off-diagonal coordinates, and this matrix is positive definite on $v^\perp$, so $\dim F_v=6-4=2$.

For $v=(1,1,1,3)^\top$, any Gram representation of a feasible matrix has unit vectors satisfying $u_1+u_2+u_3+3u_4=0$. Equality in the triangle inequality forces $u_1=u_2=u_3=-u_4$. Hence the face is the single cut matrix associated with $(1,1,1,-1)^\top$, while the dual kernel still has dimension three. Here the face equals its cut hull but strict complementarity fails. This example explains why an algorithm cannot always test the whole dual kernel.
\end{example}

Table~\ref{tab:invariants} collects these examples. Cut-span dimension concerns the available sign vectors, dual nullity concerns all directions allowed by the certificate, and face dimension concerns the remaining freedom among feasible matrices. Even when the first two agree, they need not equal the third.

\begin{table}[htbp]
\centering
\begin{tabular}{@{}lccc@{}}
\toprule
Instance & $c(W)$ & $n-d(W)$ & $\dim F^*(W)$\\
\midrule
Unit-weight $K_4$ & $3$ & $3$ & $2$\\
Unit-weight $K_6$ & $5$ & $5$ & $9$\\
Product weights, $v=(1,1,2,2)$ & $2$ & $3$ & $2$\\
Product weights, $v=(1,1,1,3)$ & $1$ & $3$ & $0$\\
\bottomrule
\end{tabular}
\caption{Cut-span dimension, dual nullity, and affine optimal-face dimension. Product weights mean $W_{ij}=4v_iv_j$ for $i\ne j$, with zero diagonal. The first three instances are strictly complementary; the last has a single optimal cut matrix but a three-dimensional dual kernel. All entries are established in the examples above.}
\label{tab:invariants}
\end{table}
\FloatBarrier

\section{Recognizing the cut hull}
\label{sec:recognition}
The geometric characterization leaves a computational question: how can we recognize a cut-generated face without being given its cuts? We separate the task into finding the subspace used by optimal matrices and recovering a Schur-independent sign basis of that subspace. Strict complementarity solves the first task by identifying the subspace with $\ker Z^*$. The main algorithm in this section solves the second using rational linear algebra and univariate polynomial factorization over $\mathbb Q$.

The input is a rational nonnegative weight matrix $W$, promised to define an exact instance; $\mathsf{CH}$ asks whether $F^*(W)=P^*(W)$. We compute $Z^*$ using Proposition~\ref{prop:slackrecovery}, so it is not additional input. Correctness is required on the promise domain. Recognizing exactness itself is a separate problem \cite{BhardwajHardness}.

\subsection{The effective support}
The dual kernel lists the directions allowed by the certificate. Some may be used by no optimal matrix. The effective support retains precisely the directions used by optimal matrices:
\begin{equation}\label{eq:effective}
\Keff=\spn\{\range X:X\in F^*(W)\}\subseteq\ker Z^*.
\end{equation}
Choose finitely many optimal matrices whose ranges together span $\Keff$; this is possible because the ambient space is finite-dimensional. Their average is still optimal and has range $\Keff$: for positive semidefinite matrices, the kernel of a sum is the intersection of their kernels. In particular, an optimal matrix can be chosen positive definite on this support. Moreover,
\[
F^*(W)=\{X\in\Ell:\range X\subseteq\Keff\}.
\]
One inclusion is the definition of $\Keff$, and the other follows from $\Keff\subseteq\ker Z^*$. This is the subspace on which the cut-face characterization applies without any additional promise.

\begin{proposition}\label{thm:effective}
For an exact instance, $F^*(W)=P^*(W)$ if and only if $\Keff$ has a Schur-independent sign basis. If $S=[s_1\ \cdots\ s_r]$ is such a basis, then
\[
F^*(W)=\conv\{s_js_j^\top:1\le j\le r\},
\qquad \mathcal C(W)=\{\pm s_1,\ldots,\pm s_r\}.
\]
\end{proposition}
\begin{proof}
If the face is generated by cuts, the classical cut-face theorem makes it a simplex with Schur-independent vertex signs \cite{LP96,KuengTropp21}. Their span is $\Keff$, so they form the required basis. Conversely, for such a basis matrix $S$, every optimum has the form $SRS^\top$. Schur independence and the unit-diagonal equations force $R$ to be diagonal with trace one, giving the displayed simplex and its rank-one points.
\end{proof}

\subsection{Recovering a Schur-independent sign basis}
\label{sec:signbasis}
Suppose a subspace has a Schur-independent sign basis, but we are given only an arbitrary rational basis $U$. In sign-basis coordinates, the constant-diagonal condition forces a symmetric coefficient matrix to be diagonal. In the supplied coordinates, the same condition therefore describes a transformed diagonal space. This space can be computed by linear equations, without knowing any signs. The question is how to recover its hidden coordinate axes.

\begin{theorem}[Exact sign-basis recognition]\label{thm:recognition}
Given a rational basis matrix for a subspace $K\subseteq\mathbb R^n$, one can decide in deterministic polynomial time whether $K$ has a Schur-independent sign basis, and recover that basis when it exists. No promise that $K$ contains a sign vector is required.
\end{theorem}

Kueng and Tropp's randomized recovery algorithm assumes $B=SW^\top$, with Schur-independent sign columns in $S$ and full-column-rank real $W$ \cite[Theorem~5.1]{KT19II}. Here there is no factorization promise: the input is an arbitrary rational subspace. We give deterministic choices, polynomial bit bounds, and a final verification that rejects negative instances. The algorithm requires no SDP solver.

The recovery mechanism uses two members of that space. If the unknown change of basis is $T$, they have the form $A=TD_1T^\top$ and $B=TD_2T^\top$, with diagonal $D_1,D_2$. When $A$ is invertible, the product $BA^{-1}=TD_2D_1^{-1}T^{-1}$ has the columns of $T$ as eigenvectors. Distinct diagonal ratios separate these lines; multiplying them by $U$ recovers the sign directions. The proof below makes both choices by bounded integer searches and verifies the resulting signs, so the algorithm also rejects subspaces with no such basis.

\begin{lemma}\label{lem:structure}
Let $U\in\mathbb Q^{n\times q}$ have full column rank, with $q\ge1$, and define
\begin{equation}\label{eq:M}
\mathcal M_U=\{R\in\mathbb S^q:\diag(URU^\top)\in\spn\{\one\}\}.
\end{equation}
If $S=UT$ is a Schur-independent sign basis matrix, then $T\in\mathbb Q^{q\times q}$ is invertible and
\begin{equation}\label{eq:diagonalspace}
\mathcal M_U=\{T\Diag(\lambda)T^\top:\lambda\in\mathbb R^q\}.
\end{equation}
In particular, $\dim\mathcal M_U=q$.
\end{lemma}
\begin{proof}
The change-of-basis matrix $T$ is invertible, and it is rational because $U$ is rational and $S$ has integer entries. Write $R=TDT^\top$. If $s_j$ is the $j$th column of $S$, then
\begin{equation}\label{eq:diagidentity}
\diag(URU^\top)=(\operatorname{tr} D)\one+2\sum_{j<k}D_{jk}(s_j\circ s_k).
\end{equation}
Schur independence says that this vector is constant precisely when $D_{jk}=0$ for every $j<k$. This is exactly~\eqref{eq:diagonalspace}.
\end{proof}

\begin{proof}[Proof of Theorem~\ref{thm:recognition}]
We assume $n\ge1$. For $K=\{0\}$, we accept the empty basis; its Schur-independence test consists of the single nonzero vector $\one$. Otherwise choose a rational full-column-rank basis matrix $U\in\mathbb Q^{n\times q}$, where $q\ge1$. The algorithm consists of the following four steps.

\smallskip\noindent\textbf{Step 1: Compute the constant-diagonal matrix space.}
Writing $u_i^\top$ for the $i$th row of $U$, the space in~\eqref{eq:M} is defined by the rational linear equations
\[
u_i^\top Ru_i-u_1^\top Ru_1=0,\qquad 2\le i\le n.
\]
Compute its dimension by Gaussian elimination. Reject if it is not $q$, as justified by Lemma~\ref{lem:structure}. Otherwise choose a rational basis $M_1,\ldots,M_q$ and set
\begin{equation}\label{eq:pencil}
M(t)=\sum_{i=1}^q t^{i-1}M_i.
\end{equation}

\smallskip\noindent\textbf{Step 2: Find an invertible member.}
Test the integers $a=0,1,\ldots,q(q-1)$ and retain the first one for which $A=M(a)$ is invertible. Reject if none exists. To see that rejection cannot occur on a yes-instance, use~\eqref{eq:diagonalspace} to write
\begin{equation}\label{eq:polydiagonal}
M(t)=T\Diag(p_1(t),\ldots,p_q(t))T^\top.
\end{equation}
The $p_j\in\mathbb Q[t]$ have degree at most $q-1$ and are linearly independent: the coefficient matrix is invertible because the $M_i$ form a basis of the diagonal-congruence space. In particular, every $p_j$ is nonzero, and
\[
\det M(t)=(\det T)^2\prod_{j=1}^q p_j(t)
\]
is a nonzero polynomial of degree at most $q(q-1)$. It cannot vanish at all the tested integers.

\smallskip\noindent\textbf{Step 3: Separate the eigenlines.}
For integers
\[
b=0,1,\ldots,(q-1)\binom q2,
\]
form $C_b=M(b)A^{-1}$ and test whether its characteristic polynomial is square-free. This can be checked by a polynomial gcd with its derivative. Retain the first successful $C_b$, denoted $C$, and reject if none exists.

On a yes-instance,~\eqref{eq:polydiagonal} gives
\begin{equation}\label{eq:eigenlines}
C_b=T\Diag\left(\frac{p_1(b)}{p_1(a)},\ldots,\frac{p_q(b)}{p_q(a)}\right)T^{-1}.
\end{equation}
All denominators are nonzero by the choice of $a$. For $j<k$, a collision of the corresponding eigenvalues occurs only at a root of
\[
p_j(t)p_k(a)-p_k(t)p_j(a).
\]
This polynomial is nonzero, since otherwise $p_j$ and $p_k$ would be proportional. It has degree at most $q-1$. Taking the union over all pairs excludes at most $(q-1)\binom q2$ integers, so one of the tested matrices has distinct eigenvalues. When $q=1$, both searches consist of the single integer zero, and the same argument applies.

\smallskip\noindent\textbf{Step 4: Recover and verify the sign basis.}
Factor the characteristic polynomial of $C$ over $\mathbb Q$. Reject unless every factor is linear. Compute a nonzero rational eigenvector $v_j$ for each of its $q$ distinct eigenvalues. For each $j$, put $y_j=Uv_j$. Reject if its first coordinate is zero; otherwise normalize $x_j=y_j/(y_j)_1$. Accept precisely when every $x_j$ is a sign vector, the matrix $[x_1\ \cdots\ x_q]$ has rank $q$, and the vectors $\one,x_j\circ x_k$ for $j<k$ have rank $1+\binom q2$.

\smallskip\noindent\emph{Correctness.} On a yes-instance,~\eqref{eq:eigenlines} has rational eigenvalues and its eigenlines are exactly the column lines of $T$. Thus Step~4 recovers the sign basis up to column permutation and sign changes. This proves completeness. On an arbitrary input, acceptance produces a directly checked Schur-independent sign basis of $K$, proving soundness. In particular, it is harmless that Steps~2 and~3 may reject other types of matrix spaces: only yes-instances are required to pass them.

\smallskip\noindent\emph{Bit complexity.} All linear systems have polynomial size in $n$ and $q$. Rational elimination yields basis matrices of polynomial bit length. The tested integers have polynomial magnitude, and their powers in~\eqref{eq:pencil} have polynomial bit length. Determinants, inverses, characteristic polynomials, polynomial gcds, and rational eigenspaces can therefore be computed with polynomial bit complexity. Factoring a univariate polynomial over $\mathbb Q$ has a deterministic polynomial-time algorithm \cite{LLL82}. There are $O(q^2)$ invertibility tests and $O(q^3)$ square-free tests, completing the complexity proof.
\end{proof}

\subsection{Recognition under strict complementarity}
\label{sec:strict}
We now have an algorithm for the sign-basis task. Under strict complementarity, the effective support is the computable dual kernel, so the two tasks combine. This hypothesis is not automatic: for general symmetric objectives, strict complementarity holds in the interior of a cut vertex's normal cone, where the optimal face is a singleton, but fails generically on its boundary \cite[Theorem~19]{deCarliSilvaTuncel19}.

\begin{corollary}\label{cor:strict}
Under the additional promise that a strictly complementary optimal pair exists, $\mathsf{CH}$ is decidable in deterministic polynomial time. On a yes-instance, all optimal cut matrices are recovered in polynomial time. The same conclusion holds if a rational basis of $\Keff$ is supplied instead of the strict-complementarity promise.
\end{corollary}
\begin{proof}
Strict complementarity is equivalent to $\Keff=\ker Z^*$. Indeed, a strictly complementary optimum has full rank on the dual kernel, and conversely a finite average realizing $\Keff=\ker Z^*$ is strictly complementary. First recover $Z^*$ from $W$ by Proposition~\ref{prop:slackrecovery}. Compute a rational basis of $\ker Z^*$ and apply Proposition~\ref{thm:effective} and Theorem~\ref{thm:recognition}. If a rational basis of $\Keff$ is supplied instead, apply those results directly to that basis.
\end{proof}

The promise cannot simply be omitted from this argument. For $v=(1,1,1,3)$, Section~\ref{sec:faces} shows that $F_v$ is a single cut matrix even though $v^\perp$ has dimension three and its sign vectors span only a line. Testing the whole dual kernel would reject this equality instance. Nor does a short rational primal optimum certify its own maximal rank: what remains to be proved is that every other optimum has range in the same support.

The failure of strict complementarity is nevertheless harmless at dual rank one: the exceptional faces can be identified directly.

\begin{corollary}\label{cor:rankone-recognition}
Let $W$ be an exact nonnegatively weighted instance without isolated vertices and $Z^*=vv^\top$ with $v>0$. Then
\[
F^*(W)=P^*(W)
\quad\Longleftrightarrow\quad
2\max_i v_i=\sum_i v_i
\quad\text{or}\quad
n=4\text{ and }v_1=v_2=v_3=v_4.
\]
For rational weights, recognition and recovery of all optimal cut matrices on equality instances take deterministic polynomial time, without a strict-complementarity promise.
\end{corollary}
\begin{proof}
A kernel sign vector gives $2\max_i v_i\le\sum_i v_i$. If equality holds at $i$, any optimal Gram representation by unit vectors satisfies
$ v_i u_i=-\sum_{j\ne i}v_j u_j$.
Equality in the triangle inequality forces all $u_j$, $j\ne i$, to coincide and $u_i$ to be their negative. The face is therefore a single cut matrix. This includes $n=2$.

Otherwise $v$ is strictly balanced. The Laurent--Poljak criterion gives an optimal matrix of rank $n-1$ \cite[Proposition~22]{deCarliSilvaTuncel19}: apply it to $b=\Diag(x)v\perp\one$ for a kernel sign vector $x$, then switch back. If $F^*=P^*$, a convex combination of optimal cut matrices of rank $n-1$ forces their sign vectors to span $v^\perp$. Hence $c(W)=n-1$, and Theorem~\ref{thm:hull} gives $1+\binom{n-1}{2}\le n$, so $n\le4$. At $n=3$, every kernel sign would give a balance equality. At $n=4$, strict balancedness excludes one-versus-three partitions. Three independent kernel signs therefore require all three two-versus-two partitions; their equations force $v_1=v_2=v_3=v_4$. Conversely, the uniform $K_4$ has the triangular optimal face described in Section~\ref{sec:faces}.

For rational input, recover $Z^*$ by Proposition~\ref{prop:slackrecovery} and test the displayed conditions on its positive rational column $Z^*e_1$, which is proportional to $v$. No square roots are needed. The equality cases yield either the single cut separating $i$ from the other vertices or the three balanced cuts of $K_4$.
\end{proof}

\section{Embedding a fitting problem beside a unique cut}
\label{sec:arithmetic}
We first build the geometry that makes recognition difficult: a prescribed cut is always the unique optimal cut, but additional SDP optima exist precisely when an auxiliary feasibility problem has a solution. Figure~\ref{fig:reduction-geometry} depicts this distinction. The construction also controls the weights, which will allow us to replace the weighted graph by a simple unweighted graph in Section~\ref{sec:sdp-complexity}.

For a full-column-rank matrix $B\in\mathbb Q^{p\times d}$, where $p\ge d\ge1$, with rows $b_i^\top$, define
\[
\mathcal T(B)=\{Q\in\mathbb S^d:Q\succeq0,\ b_i^\top Qb_i=1\ (1\le i\le p)\}.
\]
This is the centered ellipsoid-fitting feasibility problem, allowing singular positive semidefinite forms. Its connection with correlation matrices is classical: $\mathcal T(B)\ne\varnothing$ precisely when there is a correlation matrix whose kernel contains $\ker B^\top$ \cite[Proposition~3.1]{SaundersonEtAl12}. This set is bounded: since $B^\top B\succ0$, every member satisfies $\langle B^\top B,Q\rangle=p$, which bounds $\operatorname{tr}Q$. We will encode its feasibility as the existence of an optimal matrix beyond a prescribed cut.

\begin{lemma}\label{lem:embedding}
From $B$ one can construct, in polynomial time, an exact rational Max-Cut instance with strictly positive edge weights, an explicit unique optimal cut, and an explicit rational nonnegative Gram factor of its optimal dual slack, such that
\[
\F=\Pcut\quad\Longleftrightarrow\quad\mathcal T(B)=\varnothing.
\]
The construction uses $n=4(p+2)$ vertices and can have positive integer weights. Moreover, these weights are bounded by $O(n^3)$, with an integer nonnegative Gram factor, if $B$ has the following properties for some integer $1\le q\le p$:
\begin{enumerate}[label=(\alph*),leftmargin=*]
\item $qB$ is integral, every entry of $B$ has magnitude at most one, every row has at most three nonzero entries, and the first row is a coordinate row;
\item an explicit left inverse $L_BB=I$ has each row either a single coefficient $1$ or two coefficients in $\{-1/2,1/2\}$.
\end{enumerate}
\end{lemma}
\begin{proof}
The construction has three parts: remove rank-one solutions from the fitting problem, realize the enlarged system inside a dual kernel, and describe the resulting optimal face.

\smallskip\noindent\emph{Excluding additional cuts.} First exclude rank-one solutions of the auxiliary fitting problem without changing feasibility. Let $A\in\mathbb Q^{(p+2)\times(d+1)}$ have rows
\[
(b_i^\top,0)\ (1\le i\le p),\qquad (0,\ldots,0,1),\qquad (\tfrac35 b_1^\top,\tfrac45).
\]
The matrix $A$ has full column rank. If $Q\in\mathcal T(B)$, then $Q\oplus[1]\in\mathcal T(A)$. Conversely, the leading principal block of any member of $\mathcal T(A)$ belongs to $\mathcal T(B)$. Thus the two sets are nonempty simultaneously.

The set $\mathcal T(A)$ contains no rank-one matrix. Indeed, if $vv^\top$ were feasible, the measurements at $(b_1,0)$ and at the last coordinate would give numbers $\varepsilon,\delta\in\{-1,1\}$. The last measurement would require $|(3\varepsilon+4\delta)/5|=1$, whereas its possible absolute values are $7/5$ and $1/5$.

\smallskip\noindent\emph{Constructing the certificate.} The paired rows below will separate the known cut coordinate from the fitting coordinates. Their global negatives make the column space orthogonal to $\one_n$, allowing an all-ones term to make the certificate entrywise positive without changing its kernel. Put $M=p+2$, $n=4M$, and form
\[
C=\begin{pmatrix}\one_M&A\\ \one_M&-A\end{pmatrix},
\qquad
U=\begin{pmatrix}C\\-C\end{pmatrix},
\qquad
x=Ue_1=\begin{pmatrix}\one_{2M}\\-\one_{2M}\end{pmatrix}.
\]
The matrix $U$ has full column rank and $U^\top\one_n=0$. Choose a rational left inverse $L_UU=I$, and an integer $a>0$ clearing the denominators of $I-UL_U$. Define
\[
R_0=a(I-UL_U),\qquad h=1+\max_{i,j}|(R_0)_{ij}|,\qquad
N=\begin{pmatrix}\one_n^\top\\R_0+hJ_n\end{pmatrix}.
\]
Then $R_0$ is integral with $\ker R_0=\range U$, and $N$ is entrywise positive. Its first row forces zero coordinate sum, after which the remaining rows impose $R_0v=0$. Hence $\ker N=\range U$. Set $Z=N^\top N$, $W_{ij}=4Z_{ij}$ off the diagonal, and $W_{ii}=0$. This gives positive integer weights on a complete graph. Taking $s_i=Z_{ii}+\tfrac14\sum_jW_{ij}$ makes $Z$ dual feasible; $Zx=0$ proves exactness and identifies the optimal slack.

\smallskip\noindent\emph{Controlling the integer weights.} Under (a)--(b), a left inverse $L_A$ uses $L_B$ on the first $p$ rows and the new unit row for the final coordinate. To form $L_U$, recover the cut coordinate by the half-sum of a paired row $(1,a_i^\top)$ and $(1,-a_i^\top)$; recover the other coordinates by applying $L_A$ to their half-differences. Thus $L_UU=I$, its entries are multiples of $1/4$ of magnitude at most $1/2$, and each row has at most four nonzero entries. Every row of $U$ has at most four entries of magnitude at most one, and $5qU$ is integral. Therefore the choice $a=20q$ makes $R_0$ integral and gives
\[
|(UL_U)_{ij}|\le2,\qquad |(R_0)_{ij}|\le60q,\qquad
N_{ij}\le120q+1.
\]
Since $q\le p<n/4$ and $N$ has $n+1$ rows, $W_{ij}=O(nq^2)=O(n^3)$. The direct left inverse avoids determinant denominators, whose magnitudes need not be polynomial even when their bit lengths are.

\smallskip\noindent\emph{Describing the face.} Each $X\in\F$ has a unique representation $X=URU^\top$ with $R\succeq0$. Write
\[
R=\begin{pmatrix}t&h^\top\\h&H\end{pmatrix}.
\]
The rows $(1,a_i^\top)$ and $(1,-a_i^\top)$ of $C$ give
\[
t+2a_i^\top h+a_i^\top Ha_i=1,
\qquad
t-2a_i^\top h+a_i^\top Ha_i=1.
\]
Since $A$ has full column rank, these equations imply $h=0$. Consequently
\[
t\ge0,\qquad H\succeq0,\qquad a_i^\top Ha_i=1-t\quad\text{for every }i.
\]
In particular $0\le t\le1$. If $t=1$, then $AHA^\top$ is PSD with zero diagonal, hence zero; full column rank gives $H=0$. If $t<1$, then $H/(1-t)\in\mathcal T(A)$. We therefore have the exact description
\begin{equation}\label{eq:pyramid}
\F=\{xx^\top\}\ \cup\ \left\{U\bigl([1-\tau]\oplus(\tau H)\bigr)U^\top:0<\tau\le1,\ H\in\mathcal T(A)\right\}.
\end{equation}
Because $U$ has full column rank, congruence by $U$ preserves the rank of the coefficient matrix. For $0<\tau<1$ that rank is $1+\rank H$, and for $\tau=1$ it is $\rank H$. Every $H\in\mathcal T(A)$ has rank at least two, so every matrix in the second set does as well. Thus $xx^\top$ is always the only rank-one optimum, whether or not $\mathcal T(A)$ is empty. The optimal cut sign vectors are exactly $\pm x$, and~\eqref{eq:pyramid} proves the equivalence.

\smallskip\noindent\emph{Encoding size.} For general rational $B$, a left inverse is computable by rational elimination. The entries of $L_U$ and $I-UL_U$ have polynomial bit length; so do a common denominator $a$, $R_0,h,N,Z$, and $W$. Under (a)--(b), the explicit formulas above give the stronger magnitude bound. No algebraic numbers or feasible fitting matrices are computed.
\end{proof}

\begin{figure}[htbp]
\centering
\begin{minipage}[t]{0.47\linewidth}
\centering
\textbf{(a) Fitting is infeasible}\par\smallskip
$\mathcal T(B)=\varnothing$\par\medskip
\begin{tikzpicture}[every node/.style={font=\small}]
\path[use as bounding box] (-2,-0.55) rectangle (2,2.55);
\fill[cutblue] (0,2) circle (2.5pt);
\node[above=4pt] at (0,2) {$xx^\top$};
\node[align=center] at (0,0.8) {The known cut is\\the only optimal matrix.};
\end{tikzpicture}\par
$F^*(W)=P^*(W)=\{xx^\top\}$
\end{minipage}\hfill
\begin{minipage}[t]{0.49\linewidth}
\centering
\textbf{(b) Fitting is feasible}\par\smallskip
$\mathcal T(B)\ne\varnothing$\par\medskip
\begin{tikzpicture}[every node/.style={font=\small},line join=round]
\path[use as bounding box] (-2,-0.55) rectangle (2,2.55);
\fill[cutblue!7] (-1.35,0)--(0,2)--(1.35,0)--cycle;
\draw[cutblue,thick] (-1.35,0)--(0,2)--(1.35,0);
\filldraw[fill=cutteal!12,draw=cutteal] (0,0) ellipse (1.35 and 0.36);
\node at (0,0) {$\mathcal B$};
\fill[cutblue] (0,2) circle (2.5pt);
\node[above=4pt] at (0,2) {$xx^\top$};
\node[font=\footnotesize] at (0,0.65) {rank $\ge2$};
\end{tikzpicture}\par
$P^*(W)=\{xx^\top\}\subsetneq F^*(W)$
\end{minipage}
\caption{Schematic of Lemma~\ref{lem:embedding}. When fitting is feasible, the optimal face is the convex hull of the cut apex $xx^\top$ and the base $\mathcal B=\{U([0]\oplus H)U^\top:H\in\mathcal T(A)\}$, as in~\eqref{eq:pyramid}. Every point other than the apex has rank at least two, so the known cut is unique on both outcomes. The base is drawn as an ellipse only schematically; its dimension and shape depend on the fitting system.}
\label{fig:reduction-geometry}
\end{figure}

The remaining task is to produce fitting matrices with the coefficient and left-inverse bounds in Lemma~\ref{lem:embedding}, then remove the resulting integer weights without creating new optima. The explicit square-root-sum family, including its linear vertex count and strict-complementarity boundary, is retained in Appendix~\ref{ec:arithmetic}.

\section{Equivalence with exact semidefinite feasibility}
\label{sec:sdp-complexity}
The embedding in Section~\ref{sec:arithmetic} makes fitting feasibility visible in the optimal face without changing its unique cut. We now encode arbitrary exact semidefinite feasibility and remove the weights. The resulting hardness holds on connected simple graphs, so it cannot be attributed to large numerical coefficients in the Max-Cut input. The reverse direction is constructive: exact feasibility queries either recover all optimal cuts or certify that their hull is too small. Together these results identify the general computational obstruction.

Let $\SDF$ denote feasibility over the reals of a rational semidefinite system, with all matrix entries explicitly encoded in binary. We use the standard form
\begin{equation}\label{eq:sdf}
X\succeq0,\qquad\langle A_i,X\rangle=b_i\quad(1\le i\le m).
\end{equation}
Products of PSD cones and scalar linear inequalities reduce to this form by block-diagonal matrices and affine consistency equations. An affine linear matrix inequality can likewise be converted by rational elimination. These transformations have polynomial bit complexity.

Write $\CHone$ for the restriction of $\CH$ that additionally supplies a cut $x$ and promises that $xx^\top$ is the unique optimal cut matrix. It asks whether $\F=\{xx^\top\}$. Every lower reduction below produces valid promised instances on both answers. Membership in $\NP$ or $\coNP$ means that the corresponding polynomial verifier is correct on this promise domain, with no requirement on invalid inputs.

Ramana's exact alternative constructs, in polynomial time, a rational SDP feasible exactly when a given rational SDP is infeasible \cite{Ramana97}. Thus
\begin{equation}\label{eq:ramana}
\overline{\SDF}\le_m^p\SDF,
\qquad \SDF\in\NP\ \Longleftrightarrow\ \SDF\in\coNP.
\end{equation}
This uses the exact extended alternative, not the ordinary SDP dual, and includes weakly infeasible systems. It does not assume that feasible systems have short rational witnesses.

\begin{theorem}\label{thm:sdp-equivalence}
The following statements hold.
\begin{enumerate}[label=(\roman*),leftmargin=*]
\item There are polynomial-time many-one reductions
\[
\overline{\SDF}\le_m^p\CHone,
\qquad \SDF\le_m^p\CHone,
\qquad \CHone\le_m^p\SDF.
\]
The lower reductions can produce connected simple unweighted graphs with a supplied unique optimal cut. Alternatively, they can produce complete graphs with positive integer weights bounded by $O(n^3)$ and a supplied nonnegative integer Gram factor of the optimal dual slack. Complete positivity is not asserted for the unweighted slack.
\item $\CH$ is decidable by a deterministic polynomial-time algorithm with an exact $\SDF$ oracle. The algorithm uses $O(n^2)$ queries and recovers every optimal cut matrix when the answer is yes.
\item On the promise domains described above,
\[
\CH\in\PP\ \Longleftrightarrow\ \SDF\in\PP,
\]
and
\begin{equation}\label{eq:npchain}
\CH\in\NP\ \Longleftrightarrow\ \SDF\in\NP
\ \Longleftrightarrow\ \SDF\in\coNP
\ \Longleftrightarrow\ \CH\in\coNP.
\end{equation}
These equivalences also hold for $\CH$ restricted to connected simple unweighted graphs. Consequently, if either recognition problem is $\NP$-complete, then $\NP=\coNP$.
\end{enumerate}
\end{theorem}

The many-one equivalence already holds when the optimal cut is unique and supplied. For general recognition we prove an adaptive polynomial-time reduction; no many-one upper reduction is asserted. The NP-completeness consequence assumes both NP membership and NP-hardness; neither is established here.

The lower reduction first replaces large coefficients by short addition circuits, then bounds a feasible solution by encoding a small positive scale. Unit-norm gadgets convert this bounded system into a fitting matrix with a small common denominator and an explicit left inverse. Lemma~\ref{lem:embedding} then gives polynomially bounded integer weights. Finally, an unweighted replacement preserves the entire optimal face. Every feasibility step is exact; only the unique-cut embedding reverses the answer.

\subsection{Making general semidefinite feasibility bounded}
A trace bound written in binary cannot simply be assumed for a general SDP: the smallest feasible solutions can be extremely large. The following construction represents a sufficiently small positive homogenizing parameter by a polynomial number of $2\times2$ semidefinite constraints, rather than writing that parameter in binary. It uses the repeated-squaring mechanism of Khachiyan's classical large-solution example; see \cite{PatakiTouzov24}. Bounds on the magnitude of real feasible SDP solutions were established by Porkolab and Khachiyan \cite{PorkolabKhachiyan97}. We use the explicit algebraic bound of Basu and Roy below to specify a sufficient chain length.

\begin{lemma}\label{lem:bounded}
There is a deterministic polynomial-time feasibility-preserving transformation from \eqref{eq:sdf} to a system of order $D$ of the form
\begin{equation}\label{eq:bounded}
Q\succeq0,\qquad \tr Q\le1,\qquad
\langle C_i,Q\rangle=d_i\quad(1\le i\le l).
\end{equation}
The entries of $C_i$ and the numbers $d_i$ can be integers of magnitude $O(D)$.
\end{lemma}
\begin{proof}
\emph{Small coefficients.} Clear denominators in each source equality, and express its integer linear form by binary addition and subtraction circuits on the matrix entries and an anchor equal to one. A coefficient $a$ takes $O(1+\log(1+|a|))$ gates. A doubling can use a separately named equal copy, so each gate has distinct inputs and coefficients $0,\pm1$. Represent every unrestricted auxiliary scalar $z$ by nonnegative scalar blocks $z^+,z^-$ with $z=z^+-z^-$. Combine these with the original PSD variable in a block-diagonal matrix. Gate equations, off-block zeros, and the anchor equation preserve feasibility in both directions: evaluate the circuits to extend a source solution, and substitute to recover it. Converting entry equations to trace inner products and multiplying by two gives integer coefficients and constants of bounded absolute magnitude. Relabel this equivalent system as~\eqref{eq:sdf}.

\emph{A positive scale.} Obtain a sufficiently large bound on \emph{one} feasible point. Write $X=VV^\top$ with $V\in\mathbb R^{n\times n}$. Clear denominators in each equation $\langle A_i,VV^\top\rangle=b_i$ and sum the squares of the resulting integer polynomials. This gives an integer polynomial $f$ of degree at most four in $n^2$ variables and with polynomial coefficient bit length. Its real zero set is nonempty exactly when \eqref{eq:sdf} is feasible. It can be expanded in polynomial time: a fixed-degree polynomial in $n^2$ variables has only polynomially many monomials, and clearing denominators increases coefficient bit lengths by at most the sum of the input denominator lengths. In the identically zero case output the scalar bounded system $Q=0$; in the nonzero constant case output $Q=-1$, $Q\succeq0$, $\tr Q\le1$. We may assume $n\ge1$, since a zero-order source system is decided by its scalar equalities.

The explicit algebraic sampling-radius bound of Basu and Roy \cite[Theorem~2]{BasuRoy10} gives, from the degree, number of variables and coefficient bits of $f$, a polynomially computable integer $h$ of polynomial magnitude such that, if the source is feasible, it has a solution with
\begin{equation}\label{eq:hugebound}
\tr X\le 2^{2^h}-1.
\end{equation}
To compute $h$, write $\bit(a)$ for binary length, with $\bit(0)=1$. The squared-radius bound is $C2^E$, where $C,E$ have polynomial bit length at degree at most four. Choose $h=\bit(E+\bit(C)+2)$; then $C2^E<2^{2^h}-1$, and a zero $V$ in that ball has $\tr(VV^\top)=\|V\|_F^2\le C2^E$. Only $h$ is output. Remark~\ref{rem:radius} gives the explicit formulas.

Introduce a symmetric matrix $H$ and scalars $t,z_0,\ldots,z_h$. Impose
\begin{align}
&H\succeq0,\qquad \langle A_i,H\rangle=t b_i\quad(1\le i\le m),\label{eq:hom1}\\
&\tr H+t\le1,\qquad t\ge z_h,\qquad z_0=\tfrac12,\qquad 0\le z_j\le1,\label{eq:hom2}\\
&\begin{pmatrix}z_{j+1}&z_j\\z_j&1\end{pmatrix}\succeq0
\quad(0\le j<h).\label{eq:hom3}
\end{align}
These are affine semidefinite constraints. They imply $z_{j+1}\ge z_j^2$, hence
\[
t\ge z_h\ge2^{-2^h}>0.
\]
Consequently any feasible solution gives a source solution $X=H/t$.

Conversely, choose a source solution satisfying \eqref{eq:hugebound}. Set
\[
t=\frac1{1+\tr X},\qquad H=tX,\qquad z_j=2^{-2^j}.
\]
Then $t\ge z_h$ and all displayed constraints hold. These real numbers establish existence; the transformation does not output them.

The reverse implication uses only $t>0$, so infeasible systems, including weakly infeasible ones, cannot acquire a feasible solution through homogenization. The sampling bound is needed only for the forward implication.

To obtain a single PSD matrix variable, use the block diagonal matrix comprising $H$, the $2\times2$ matrices in \eqref{eq:hom3}, and scalar blocks $t$, $1-\tr H-t$, $t-z_h$, $z_j$, and $1-z_j$. Its order is $D=n+4h+5$. Impose the affine equations matching the repeated entries and all off-block zeros. The trace inequality bounds $t$ and each diagonal entry of $H$ by one. The chain variables lie in $[0,1]$, and each remaining scalar slack lies there as well. Thus every diagonal entry of this block matrix is between zero and one, so its trace is at most $D$. All affine matching equations have coefficients and constants of bounded absolute magnitude and a common denominator dividing two. Substitute the block matrix $DQ$ and clear this denominator. This gives~\eqref{eq:bounded} with integer coefficients of magnitude $O(D)$. The trace bound was already implied before normalization. All dimensions are polynomial in the source input length.
\end{proof}

\subsection{Unit-norm constraints encode bounded semidefinite systems}
We return to the fitting set $\mathcal T(B)$ of Section~\ref{sec:arithmetic}. PSD and Euclidean-distance completion prescribe Gram entries or squared distances \cite{Laurent00Completion}. Here every permitted measurement has the same form, $b_i^\top G b_i=1$. We must express arbitrary rational affine equations using only these measurements, while keeping the encoding polynomial and allowing singular $G$.

The construction works with abstract Gram vectors: their squared norms are diagonal entries of $G$, and squared norms of rational linear combinations are rank-one measurements of $G$. Fresh auxiliary vectors let us copy and add squared norms without changing the original inner products. The next lemma implements these operations. Its ``scalar coordinates'' are Gram vectors used to store scalar squared norms; they need not lie on a common line.

\begin{lemma}\label{lem:gadgets}
The following operations can be implemented by a constant number of unit-norm measurements using fresh Gram coordinates.
\begin{enumerate}[label=(\alph*),leftmargin=*]
\item Enforce $\|u\|^2\le1$ by a fresh padding coordinate $p$ and the two constraints $\|u+p\|^2=\|u-p\|^2=1$.
\item Copy a squared norm from a rational linear form $v$ into a fresh scalar coordinate $a$ using a fresh padding coordinate $p$ and
\[
\|v\pm p\|^2=1,\qquad \|a\pm p\|^2=1.
\]
These force $\|a\|^2=\|v\|^2\le1$.
\item For distinct scalar coordinates $a,b$, introduce fresh $c,p$ and impose
\[
\|\varepsilon a+\delta b+p\|^2=1\quad(\varepsilon,\delta\in\{-1,1\}),
\qquad \|c\pm p\|^2=1.
\]
These force $\|c\|^2=\|a\|^2+\|b\|^2\le1$.
\item Equal squared norms of scalar coordinates $a,c$, each at most one, can be imposed by $\|a\pm p\|^2=\|c\pm p\|^2=1$ with a fresh $p$.
\end{enumerate}
In a forward construction all auxiliary scalar coordinates can be chosen mutually orthogonal and orthogonal to the original Gram configuration. In (c), orthogonality of $a$ and $b$ is an additional enforced condition: the gadget is used on freely chosen auxiliary copies, not on arbitrary prescribed original vectors.
\end{lemma}
\begin{proof}
Adding the two equations in (a) gives $\|u\|^2+\|p\|^2=1$, and subtracting gives $\langle u,p\rangle=0$. This also proves (b) and (d). Averaging the four equations in (c) gives
\[
\|a\|^2+\|b\|^2+\|p\|^2=1;
\]
weighting them by $\varepsilon$, $\delta$ and $\varepsilon\delta$ shows that the three cross inner products vanish. The last pair gives $\|c\|^2+\|p\|^2=1$. Conversely choose each fresh scalar coordinate in a new orthogonal direction, with its required nonnegative squared norm. For (c), the padding squared norm is $1-\|a\|^2-\|b\|^2$. Thus the forward construction works whenever the stated sums are at most one.
\end{proof}

\begin{lemma}\label{lem:fitting}
From the integer-coefficient system~\eqref{eq:bounded} supplied by Lemma~\ref{lem:bounded}, one can construct in deterministic polynomial time a full-column-rank rational matrix $B$ satisfying (a)--(b) of Lemma~\ref{lem:embedding}, such that
\[
\mathcal T(B)\ne\varnothing
\quad\Longleftrightarrow\quad
\eqref{eq:bounded}\text{ is feasible}.
\]
Restriction to the original Gram coordinates maps $\mathcal T(B)$ onto the source feasible set. A rational affine lift of every source feasible matrix is available.
\end{lemma}
\begin{proof}
We first enforce the trace bound, then encode each affine equation by copying and adding squared norms.

\emph{Enforcing the trace bound.} Use original Gram coordinates $u_1,\ldots,u_D$ for $Q$, and an anchor $u_0$ with $\|u_0\|^2=1$. Add the bound gadget for each $u_i$. Copy each $\|u_i\|^2$ into a fresh scalar coordinate and add them successively with the addition gadget. All intermediate sums are at most one exactly as required by $\tr Q\le1$; add the bound gadget to the final sum coordinate. For $D=1$ this last bound is already sufficient. If $s_j$ denotes the auxiliary coordinate representing the first $j$ squared norms, the gadget identities give $\|s_j\|^2=\sum_{i\le j}\|u_i\|^2$. Bounding the last sum therefore enforces exactly the trace inequality. The copied coordinates may be orthogonal even when the original $u_i$ are not, because a copy constrains only a squared norm. No orthogonality is imposed between the original coordinates themselves.

\emph{Expressing an affine equation through squared norms.} Consider $\langle C,Q\rangle=d_0$. The identity
\begin{equation}\label{eq:polar}
\langle C,Q\rangle
=\sum_i\left(C_{ii}-\sum_{j\ne i}C_{ij}\right)\|u_i\|^2
 +\sum_{i<j}C_{ij}\|u_i+u_j\|^2
\end{equation}
expresses its left side using rational multiples of squared norms. Move $d_0\|u_0\|^2$ to the left and separate the positive and negative coefficients.

All coefficients are integers. Expand a coefficient of absolute value $c$ into $c$ occurrences, separating positive and negative sides. Every occurrence is $u_i$, $u_i+u_j$, or $u_0$. If the source order is $D$, its coefficient bound is $H_0=O(D)$, and it has $l$ equations, the total number $M$ of occurrences is at most $l(2D^2+1)H_0$. Thus this expansion is polynomial. Each equation becomes
\begin{equation}\label{eq:sum-norms}
\sum_{v\in\mathcal V_+}\|v\|^2=\sum_{v\in\mathcal V_-}\|v\|^2.
\end{equation}

\emph{One small denominator.} Let $T$ be the sum of squared Euclidean norms of the coefficient rows of all occurrences, across all equations. Then $M\le T\le2M$. Put $q=1+T$ and replace every occurrence $v$ by $v/q$. Since $Q\preceq I$, the combined squared norms of all these scaled forms are at most $T/(1+T)^2\le1/4$. This estimate does not require orthogonality between the anchor and original coordinates: no occurrence mixes the two. Every addition therefore remains bounded by one, and every coefficient is in $\{0,\pm1,\pm1/q\}$. For $M=0$, take $q=1$.

\emph{Implementing the equation.} Copy every scaled squared norm to a distinct fresh scalar coordinate. Add each list by successive addition gadgets, and impose equality of the resulting two squared norms using (d). For an empty list introduce a zero coordinate $z$ by $\|u_0\pm z\|^2=1$, which forces $\|z\|^2=0$. A list of length one needs no addition. In an addition chain, the accumulated sum coordinate and the next copied coordinate are distinct auxiliary coordinates; their forced orthogonality therefore imposes no extra condition on the original Gram matrix. A completely zero equation can instead be omitted. Repeat for every source equation.

\emph{Soundness.} Given a feasible Gram matrix for the constructed measurements, its original principal block $Q$ is PSD. The trace gadgets give $\tr Q\le1$. Copies and addition gadgets give exactly \eqref{eq:sum-norms}, hence all source equations by \eqref{eq:polar}.

\emph{Completeness and affine lift.} Given a source solution $Q$, start with the Gram matrix $[1]\oplus Q$. Give every auxiliary coordinate its prescribed squared norm and put it in a new orthogonal direction, orthogonal to all original coordinates and all other auxiliary coordinates. Every padding value is nonnegative by the trace bound and the scaling above. Thus every measurement is one. Equivalently, the lifted Gram matrix is the block diagonal sum of $[1]$, $Q$, and scalar blocks which are affine rational functions of $Q$. This proves the asserted affine lift and rationality preservation.

\emph{Size and explicit left inverse.} Let $p$ be the number of measurement rows. Each occurrence uses four copy measurements, so $p\ge1+4M\ge q$. Each row has at most three nonzero entries, and the anchor is the first row. Every coordinate has a direct expression in the rows: the original bound pairs isolate their coordinate and padding by half-sums and half-differences; the pairs $a\pm p$ do the same for a copied coordinate and its padding; the pairs $c\pm p$ isolate a sum coordinate and its padding. Equality pairs isolate their padding, and $u_0\pm z$ isolates a zero coordinate. These identities give $L_BB=I$ with each row as in Lemma~\ref{lem:embedding}(b), including on infeasible inputs. All dimensions are $O(D+M+l)$ and all coefficients have polynomial bit length. The corresponding affine-lift construction for arbitrary rational coefficients is recorded in Remark~\ref{rem:rational-fitting}.
\end{proof}

\begin{corollary}\label{cor:general-fitting}
General rational semidefinite feasibility is polynomial-time many-one reducible to feasibility of $\mathcal T(B)$ with rational full-column-rank $B$. The reverse reduction is immediate. Thus rational ellipsoid-fitting feasibility in this sense is many-one equivalent to $\SDF$. Here fitting permits singular PSD matrices; replacing $G\succeq0$ by $G\succ0$ is a different problem.
\end{corollary}
\begin{proof}
Compose Lemmas~\ref{lem:bounded} and~\ref{lem:fitting}. Each fitting system is itself a rational semidefinite system.
\end{proof}

\subsection{Removing weights without changing the optimal face}
Weighted-to-unweighted reductions for Max-Cut are classical in approximation complexity \cite{CrescenziSilvestriTrevisan01}; here we require an affine bijection of the entire SDP optimal faces, preserving ranks and optimal cut matrices. Connecting each copy set to auxiliary vertices gives a strict spectral margin that forces all copies to agree in every optimum.

\begin{lemma}\label{lem:unweight}
Let $W$ be an exact instance with nonnegative integer weights and a known optimal cut $x$. Put
\[
\Delta=\max_i\sum_jW_{ij},\qquad r=1+\max_{i,j}W_{ij},\qquad \ell=2\Delta+1.
\]
One can construct a simple graph $G$ on $n(r+\ell)$ vertices whose optimal face is affinely isomorphic to $F^*(W)$. The map preserves rank and all optimal cut matrices, and supplies an optimal cut of $G$. Connected support of $W$ gives connected $G$. Construction time is polynomial in $n$ and the numerical weights.
\end{lemma}
\begin{proof}
\emph{The graph.} For each vertex $i$, create $r$ copies $C_i$ and $\ell$ auxiliary vertices $D_i$, and join $C_i$ completely to $D_i$. Between $C_i$ and $C_j$ put a $W_{ij}$-regular simple bipartite graph, using the first $W_{ij}$ cyclic perfect matchings on the two sets of size $r$. There are no other edges. Set $\widehat x=x_i$ on $C_i$ and $\widehat x=-x_i$ on $D_i$.

\emph{The exact kernel.} For signs $z$ and weights $V$, write the switched certificate as
\[
S(V,z)=\Diag(-z\circ(Vz))+\Diag(z)V\Diag(z)=4\Diag(z)Z_z\Diag(z).
\]
The old certificate $S=S(W,x)$ is PSD. Decompose a vector $y$ on the new vertices into means $a_i$ on $C_i$, $b_i$ on $D_i$, and zero-sum deviations $\eta_i,\zeta_i$ on those sets. The internal edges cross $\widehat x$, so their contribution is
\[
r\ell\|a-b\|^2+\ell\|\eta\|^2+r\|\zeta\|^2.
\]
For the external edges, every copy of $i$ has diagonal entry $-x_i\sum_jW_{ij}x_j$ and $W_{ij}$ neighbors in $C_j$. Thus their switched certificate $E$ acts as $S$ on block means, contributing $r a^\top Sa$. Regularity and symmetry separate this invariant subspace from the zero-sum deviations. Moreover, each absolute diagonal entry and off-diagonal absolute row sum of $E$ is at most $\Delta$, so $\|E\|_2\le2\Delta$. Consequently the new switched certificate satisfies
\[
y^\top\widehat Sy\ge r a^\top Sa+r\ell\|a-b\|^2
 +(\ell-2\Delta)\|\eta\|^2+r\|\zeta\|^2\ge0.
\]
Because $\ell-2\Delta=1$, equality holds exactly when $\eta=\zeta=0$, $a=b$, and $a\in\ker S$; substituting these conditions also proves the reverse inclusion. Hence $\widehat x$ is optimal and the new instance is exact.

\emph{All optima correspond.} Define $T$ by placing $u_i$ on $C_i$ and $-u_i$ on $D_i$. Unswitching gives $\ker Z_G^*=T\ker Z_W^*$. Thus $X\mapsto TXT^\top$ lifts every old optimum to a new one. Conversely, $T^\top T=(r+\ell)I$, so $L=(r+\ell)^{-1}T^\top$ is a left inverse. For a new optimum $\widehat X$, the matrix $X=L\widehat XL^\top$ is PSD, supported on $\ker Z_W^*$, and satisfies $TXT^\top=\widehat X$. Its diagonal is one, as seen on any copy of each vertex. This proves surjectivity. Full column rank of $T$ proves injectivity and rank preservation; rank-one unit-diagonal PSD matrices are exactly cut matrices. Each $C_i\cup D_i$ is connected, and old support edges join these sets, giving the connectedness assertion.
\end{proof}

\begin{proof}[Proof of Theorem~\ref{thm:sdp-equivalence}(i), lower bounds]
Lemmas~\ref{lem:bounded} and~\ref{lem:fitting} preserve feasibility and supply the coefficient and left-inverse bounds required by Lemma~\ref{lem:embedding}. Its output has $n$ vertices, positive integer weights $O(n^3)$, an explicit integer nonnegative Gram certificate, and a unique optimal cut. Cut-hull equality holds exactly when the source is infeasible. Lemma~\ref{lem:unweight} preserves this equivalence and the unique cut on a connected simple graph. Since $\Delta=O(n^4)$, it uses $O(n^5)$ vertices; even an explicit adjacency matrix has polynomial size. This proves $\overline{\SDF}\le_m^p\CHone$ on both stated graph classes. Applying~\eqref{eq:ramana} to the source first gives $\SDF\le_m^p\CHone$. Every output is exact with its unique cut supplied, on both answers.
\end{proof}

\subsection{Recognition with an exact semidefinite oracle}
We now show that exact SDP queries suffice for recognition. Maintain a list of independent optimal signs and ask whether any optimal matrix uses a direction outside their span. If none does, Schur independence certifies that the list generates the face. Otherwise, try to obtain a new cut by fixing its first-row signs one at a time. If the face is cut-generated, some sign choice must remain feasible at every step; failure therefore certifies a no-instance.

The cut-face characterization from Section~\ref{sec:faces} ensures that every subset of distinct optimal cut vertices is Schur independent on a yes-instance. Every cut matrix in a face is indeed a vertex: in a PSD decomposition of $xx^\top$, each summand is supported on $\spn\{x\}$, and unit diagonal forces it to equal $xx^\top$.

For independent known optimal signs collected as the columns of $S$, define
\[
Q_S=I-S(S^\top S)^{-1}S^\top;
\]
when $S$ is empty, set $Q_S=I$. For $X\succeq0$,
\begin{equation}\label{eq:outside}
\langle Q_S,X\rangle>0
\quad\Longleftrightarrow\quad
\range X\not\subseteq\range S.
\end{equation}
Every query below consists of $X\in\F$, specified constraints $X_{1j}=\varepsilon_j\in\{-1,1\}$, and \eqref{eq:outside}. The strict inequality is represented exactly by an ordinary semidefinite constraint:
\begin{equation}\label{eq:strict}
\exists t\in\mathbb R:\quad
\begin{pmatrix}\langle Q_S,X\rangle&1\\1&t\end{pmatrix}\succeq0.
\end{equation}
This is equivalent to positivity of the upper-left entry, by the determinant condition. A potentially very large $t$ is not computed; the query merely specifies a rational SDP.

\begin{lemma}\label{lem:endpoints}
Suppose $\F$ is generated by cut matrices. If the constraints $X\in\F$, a specified first-row sign prefix, and \eqref{eq:outside} are feasible, then some optimal cut matrix satisfies all those constraints. Consequently, for each next first-row coordinate at least one of the two endpoints $-1,+1$ preserves feasibility.
\end{lemma}
\begin{proof}
Write a feasible $X$ as a convex combination of optimal cut matrices, using positive coefficients. If $X_{1j}=\varepsilon_j\in\{-1,1\}$, all participating cuts have that entry equal to $\varepsilon_j$, since each summand entry is $-1$ or $+1$. Further, the objective $\langle Q_S,\cdot\rangle$ is nonnegative on every PSD summand. Its positive value on $X$ implies a positive value on at least one participating cut. That cut satisfies the whole prefix and the outside-support condition, proving both assertions.
\end{proof}

\begin{theorem}\label{thm:oracle}
There is a deterministic polynomial-time $\SDF$-oracle algorithm for $\CH$. It uses $O(n^2)$ feasibility queries. On a yes-instance, it returns all optimal cuts up to global sign.
\end{theorem}
\begin{proof}
First compute $Z^*$ by Proposition~\ref{prop:slackrecovery}, so that $X\in\F$ is given by the rational constraints $X\succeq0$, $\diag X=\one$, and $Z^*X=0$. Start with an empty list $S$. Throughout the algorithm, its columns are independent optimal signs. During a prefix search, feasibility of the current prefix together with the outside-support condition is the invariant. Repeat the following steps.
\begin{enumerate}[leftmargin=*]
\item Test Schur independence of the current list by rational rank computation. If it fails, return no. The empty list passes.
\item Ask whether $X\in\F$ and \eqref{eq:outside} are feasible, using \eqref{eq:strict}. If not, return yes with the list $S$.
\item Otherwise initialize an empty first-row prefix. For $j=2,\ldots,n$, test feasibility with $X_{1j}=-1$ added to the accumulated prefix and the outside-support condition. If feasible, retain $-1$. If infeasible, test $+1$ instead and retain it if feasible. If neither endpoint is feasible, return no.
\item After all $n-1$ choices, put $x_1=1$ and $x_j=X_{1j}$ as chosen. Append $x$ to $S$ and repeat.
\end{enumerate}

All successful endpoint choices preserve feasibility. A PSD unit-diagonal matrix whose first-row entries are all signs must be $xx^\top$: in a unit-vector Gram representation, equality in Cauchy--Schwarz gives $v_j=x_jv_1$. Hence the vector appended in step~4 is an optimal cut, and \eqref{eq:outside} makes it independent of the previous columns. There are at most $n$ append operations. Notice that the algorithm needs only a yes/no answer to each query: after a successful prefix search, the signs themselves specify the cut, so no real feasible matrix needs to be extracted from the oracle.

\emph{A no answer is sound.} If the list fails Schur independence, the classical cut-face characterization rules out $\F=\Pcut$. If an endpoint extension fails on both signs, Lemma~\ref{lem:endpoints} rules out that equality. These are the only ways to return no.

\emph{A yes answer is sound and returns all cuts.} When step~2 returns no, every $X\in\F$ is supported on $\range S$. The list cannot be empty, since $\tr X=n$ on the nonempty face. Thus write $X=SRS^\top$ with $R\succeq0$. Schur independence and
\[
\diag(SRS^\top)=(\tr R)\one+2\sum_{i<j}R_{ij}(s_i\circ s_j)=\one
\]
force $R$ to be diagonal with trace one. Consequently $\F$ is exactly the convex hull of the listed optimal cut matrices. Since $S$ has independent columns, the rank of $S\Diag(\lambda)S^\top$ is the number of positive coefficients. Its only rank-one members are the listed vertices, proving complete recovery.

\emph{Termination and encoding.} Each successful iteration adds an independent sign, and when $S$ spans $\mathbb R^n$ step~2 is infeasible. Each iteration uses one support query and at most $2(n-1)$ endpoint queries. The Gram matrix $S^\top S$ has integer entries of magnitude at most $n$, so its inverse and $Q_S$ have polynomial bit length. Every query is a polynomial-size rational SDP; all remaining work is polynomial-time rational linear algebra. No bound on the bit length of a feasible answer to an oracle query is used.
\end{proof}

\begin{proof}[Proof of Theorem~\ref{thm:sdp-equivalence}(i), upper bound for the unique-cut restriction]
For the supplied unique optimal cut $x$, Proposition~\ref{prop:certificate} gives the optimal slack $Z^*=Z_x=\Diag(W(x)/2)-L_W/4$. Put $Q_x=I-xx^\top/n$. The condition $\F=\{xx^\top\}$ holds exactly when no $X\in\F$ has $\langle Q_x,X\rangle>0$. Indeed, a feasible PSD unit-diagonal matrix supported on $\spn\{x\}$ is exactly $xx^\top$. Encode positivity by \eqref{eq:strict}. Thus $\CHone\le_m^p\overline{\SDF}$, and the exact-alternative transformation gives $\CHone\le_m^p\SDF$.
\end{proof}

\subsection{Complexity consequences}
\begin{proof}[Proof of Theorem~\ref{thm:sdp-equivalence}(ii)--(iii)]
Part (ii) is Theorem~\ref{thm:oracle}. The many-one lower bound and the oracle upper bound give the equivalence of polynomial-time solvability.

If $\CH\in\NP$, then $\SDF\in\NP$ by the many-one reduction; by \eqref{eq:ramana}, $\SDF\in\NP\cap\coNP$. For every language $L\in\NP\cap\coNP$,
\[
\PP^L\subseteq\NP\cap\coNP.
\]
For clarity, an accepting computation of a deterministic polynomial-time oracle machine can be certified by listing its polynomially many queries, their answers, and an appropriate NP certificate for each yes or no answer. The verified answers determine the adaptive computation. The same argument certifies a rejecting computation. To see soundness with adaptive queries, the first certified answer must be correct, and induction then forces every later query and answer to agree with the actual deterministic computation. The total certificate length is polynomial because there are polynomially many queries of polynomial length. This proves the inclusion, including for adaptive queries. Applying it to Theorem~\ref{thm:oracle} puts $\CH$ in both classes whenever $\SDF$ belongs to either one. Conversely, membership of $\CH$ in $\coNP$ gives membership of $\SDF$ in $\coNP$ by the lower bound, and then in $\NP$ by \eqref{eq:ramana}. This proves \eqref{eq:npchain}.

The same arguments apply to recognition restricted to connected simple graphs: the lower reduction lands in that class and the oracle upper bound remains valid. Finally, suppose $\CH$ is $\NP$-complete, with reductions producing valid promised inputs. Its NP membership then implies its coNP membership by \eqref{eq:npchain}. Every NP language reduces to it, so $\NP\subseteq\coNP$, and taking complements gives equality.
\end{proof}

\begin{corollary}\label{cor:uniqueness}
Let $\mathsf{UNIQ}$ consist of pairs $(W,x)$ with rational symmetric nonnegative zero-diagonal $W$ and $x\in\sgn^n$, for which $xx^\top$ is the unique optimal SDP solution. Let $\mathsf{UNIQ}_0$ restrict $W$ to adjacency matrices of simple graphs. Neither language imposes an exactness or optimality promise, and isolated vertices are allowed. Then
\[
\mathsf{UNIQ}_0\equiv_m^p\mathsf{UNIQ}\equiv_m^p\SDF.
\]
The lower reduction to $\mathsf{UNIQ}_0$ produces connected graphs whose supplied cut is always the unique optimal cut up to sign. The weighted reduction has the integer-weight and Gram-factor properties in Theorem~\ref{thm:sdp-equivalence}(i).
\end{corollary}

\begin{proof}
Test $Z_x\succeq0$ by exact rational symmetric elimination. If the test fails, Proposition~\ref{prop:certificate}, whose certificate equivalence remains valid with isolated vertices, shows that $xx^\top$ is not optimal; output a fixed infeasible SDP. Otherwise, the upper-bound proof of Theorem~\ref{thm:sdp-equivalence}(i) tests whether $\F=\{xx^\top\}$ by an exact-alternative transformation of the outside-support query with $Q_x=I-xx^\top/n$. This argument requires only that $x$ be optimal, not that it be the unique optimal cut, and gives $\mathsf{UNIQ}\le_m^p\SDF$. Conversely, the lower reductions in that theorem supply a unique optimal cut, so $\F=\Pcut$ is equivalent to $\F=\{xx^\top\}$ in both graph classes. Restricting the upper reduction to adjacency matrices gives the assertion for $\mathsf{UNIQ}_0$.
\end{proof}

\subsection{Certificates over the reals}
\label{sec:real-certificates}
The class $\exists\mathbb R$ consists of problems reducible in polynomial time to real polynomial feasibility; $\mathrm{co}\exists\mathbb R$ is its complement class. The oracle formulation also yields existential-real certificates for both outcomes. It is useful to state this separately because the number of real variables in a certificate and the binary length of a rational witness are different issues.

\begin{corollary}\label{prop:c1-real-complexity}
On its promise domain, $\CH$ belongs to $\exists\mathbb R\cap\mathrm{co}\exists\mathbb R$.
\end{corollary}
\begin{proof}
Feasibility and infeasibility of each rational SDP query both have polynomial-size existential-real certificates, using Gram factors and Ramana's exact alternative. Encode the polynomial-time oracle computation by a Boolean circuit, and attach the certificate selected by each oracle-answer bit. Requiring an accepting or rejecting output yields the two formulations. Appendix~\ref{ec:real} gives the uniform encoding and its soundness proof, including adaptive queries and weak infeasibility.
\end{proof}

On an equality instance, the recovered sign matrix $S$ with $r$ columns gives the rational barycentre $SS^\top/r$, whose entries have $O(\log(r+1))$ bits. The remaining difficulty is to certify that every other optimum has range in $\range S$. This explains why a short rational primal optimum alone does not imply an ordinary NP certificate.

\section{Extremal cut spans and orthogonal cuts}
\label{sec:orthogonal}
We return to the spanning question at its extremal value $c(W)=n-1$. This forces dual rank one but permits nonuniform weights. Orthogonality forces uniformity and a Hadamard matrix; without it, a determinant bound still leaves only finitely many positive weight directions at each order.

For example, $v=(1,1,1,1,1,3)$ gives product weights taking the two values $4$ and $12$, yet $c(W)=5$. The five kernel signs whose final coordinate is $+1$ have exactly one positive entry among the first five coordinates. A linear relation with coefficients $a_i$ and sum $a$ gives $2a_i-a=0$ in those coordinates and $a=0$ in the last, proving independence. Thus nonuniformity is compatible with maximal cut-span.

\subsection{Orthogonal optimal cuts}
Write $o(W)$ for the maximum size of a pairwise orthogonal family of optimal cut signs. A Hadamard matrix of order $n$ is a sign matrix $H$ satisfying $HH^\top=nI$.

An orthogonal sign family is linearly independent, so $o(W)\le c(W)\le n-d(W)\le n-1$. Attaining the last bound has a much stronger consequence. The proof uses the elementary completion of $n-1$ orthogonal sign rows to a Hadamard matrix; see \cite{horadam2007hadamard,GoldbergerStrassler22} for the surrounding theory.

\begin{theorem}\label{thm:hadamard}
For an exact instance with nonnegative weights and no isolated vertices, $o(W)\le n-1$. Equality holds if and only if a Hadamard matrix of order $n$ exists and $W=\lambda A(K_n)$ for some $\lambda>0$.
\end{theorem}
\begin{proof}
Suppose the bound $o(W)\le c(W)\le n-d(W)\le n-1$ is attained. Then $d(W)=1$ and Proposition~\ref{thm:rankone} gives $Z^*=vv^\top$ with $v>0$. The $n-1$ orthogonal optimal signs form a basis of $v^\perp$, each with squared norm $n$. Append the row $\sqrt n\,v^\top/\|v\|$ to obtain a square matrix $H$. Since $HH^\top=nI$, also $H^\top H=nI$, whose diagonal equations give
\[
n-1+n\frac{v_i^2}{\|v\|^2}=n\qquad(1\le i\le n).
\]
Positivity of $v$ makes all coordinates equal, so the weights $4v_iv_j$ are uniform and the appended row is $\one^\top$. Thus $H$ is Hadamard.

Conversely, normalize one row of a Hadamard matrix to $\one^\top$ by column sign changes. Its other rows are balanced, pairwise orthogonal signs and hence optimal cuts of uniform $K_n$ by Example~\ref{prop:completeface}. Positive rescaling preserves exactness and the optimal cuts.
\end{proof}

Maximal span and orthogonality already differ for $K_6$: its balanced signs have pairwise inner product congruent to two modulo four, so $c(K_6)=5$ but $o(K_6)=1$. Further attainment results by order are proved in Appendix~\ref{ec:orthogonal}.

\subsection{Positive normal directions}
Without orthogonality, maximal cut-span still forces product weights by Proposition~\ref{thm:rankone}. Their positive factor $v$ is determined up to scaling by the hyperplane spanned by the optimal signs. Cofactors show that this direction is rational, and maximal binary determinants bound its primitive integer representative. Related integer-normal questions are studied in \cite{AlonVu97}.

An integer vector is primitive when its coordinates have greatest common divisor one. Define
\[
D_{01}(j)=\max\{|\det M|:M\in\{0,1\}^{j\times j}\},\qquad D_{01}(0)=1.
\]

\begin{proposition}\label{prop:cofactor}
Suppose that $n\ge2$, $v\in\mathbb R_{>0}^n$, and $v^\perp$ is spanned by sign vectors. Then the positive normal direction has a unique primitive integer representative $p\in\mathbb Z_{>0}^n$, and
\begin{equation}\label{eq:detbound}
\max_i p_i
\le D_{01}(n-2)
\le
\frac{(n-1)^{(n-1)/2}}{2^{n-2}}.
\end{equation}
Consequently, for each fixed $n$, only finitely many such positive directions exist, even before identifying vertex permutations.
\end{proposition}

\begin{proof}
For $n=2$, a sign vector orthogonal to $v>0$ forces $v_1=v_2$, so the primitive representative is $(1,1)$ and the bound holds. Suppose $n\ge3$, and form a matrix $A\in\{-1,1\}^{(n-1)\times n}$ from independent signs spanning $v^\perp$. Its signed maximal minors $\eta_i=(-1)^i\det A_{\widehat i}$ form a nonzero integer vector in $\ker A=\spn\{v\}$. Thus all coordinates are nonzero and have a common sign. Choose that sign positive and divide by $g=\gcd(\eta_1,\ldots,\eta_n)$ to obtain the unique positive primitive representative $p$.

For any sign matrix $M$ of order $n-1$, row and column sign changes make its first row and column all ones without changing its absolute determinant. Subtracting the first row from each other row and expanding along the first column then gives $|\det M|=2^{n-2}|\det N|$ for a binary matrix $N$ of order $n-2$. Consequently every $\eta_i$ is divisible by $2^{n-2}$, and
\[
g\ge2^{n-2},\qquad |\eta_i|\le2^{n-2}D_{01}(n-2),\qquad p_i\le D_{01}(n-2).
\]

Conversely, every binary matrix $N$ of order $n-2$ can be bordered to give
\[
M_N=\begin{pmatrix}1&\one^\top\\\one&J_{n-2}-2N\end{pmatrix},
\qquad |\det M_N|=2^{n-2}|\det N|.
\]
Each row of $M_N$ has norm $\sqrt{n-1}$, so Hadamard's inequality bounds this determinant by $(n-1)^{(n-1)/2}$. Maximizing over $N$ proves the second inequality. The coordinate bound leaves finitely many primitive representatives.
\end{proof}

Attainment requires cofactors of a single sign matrix with compatible signs and a large coordinate after division by their common divisor. Exact enumeration gives $1,0,1,1,4,14,122$ primitive positive directions up to permutation for $n=2,\ldots,8$; every nonempty case attains $D_{01}(n-2)$. Appendix~\ref{sec:computations} gives the exact enumeration procedures, tables, and the separate census of small unweighted graphs.

\section{Conclusion and further questions}
\label{sec:discussion}
In this work, we studied two ways in which optimal cuts can determine an exact Max-Cut SDP: by spanning the dual kernel and by generating the entire optimal face. Replication makes the first property available for every doubly nonnegative certificate with positive diagonal and rational kernel, preserving rank and nonnegative factorization properties. Recognition of the second property depends on the effective primal support. The optimal dual slack itself can be recovered from rational weights in polynomial time under exactness. When strict complementarity identifies the effective support with its kernel, recognition is polynomial-time, as is cut recovery on equality instances. In general, recognition is polynomial-time Turing equivalent to exact SDP feasibility, even on connected simple graphs. Testing uniqueness of a supplied cut optimum is many-one equivalent, without an exactness promise. The unweighted construction preserves the entire optimal face, showing that this difficulty does not depend on numerical weights.

Several questions remain. The replication threshold may admit sharper bounds from the integer geometry of the kernel, and nonuniform replication can reduce the required order. Realization on a fixed support is more restrictive than realization after enlargement; in particular, the existence of a simple unweighted exact graph with a non-completely-positive sign-spanned rank-three slack remains open here. The determinant bound for positive normal directions is attained in every nonempty case through order eight, but its sharpness at larger orders is unresolved. On the algorithmic side, assumptions weaker than strict complementarity that make the effective support accessible would extend the tractable case. The general classification ties ordinary NP membership to the corresponding open issue for exact SDP feasibility; it also leaves open whether unrestricted recognition admits a many-one upper reduction to that problem.

\appendix

\section{Binary aggregation and integrality}
\subsection{Binary aggregation}\label{ec:binary}
The separation into within-block directions and block sums does not depend on the kernel equation or on complement symmetry. The same argument gives a dimension identity for any binary set with this permutation symmetry.

\begin{proposition}\label{prop:binarydimension}
Let $F\subseteq\{0,1\}^n$ be nonempty and invariant under independent permutations within blocks $V_1,\ldots,V_r$ of sizes $m_1,\ldots,m_r$. Let $\pi$ record block sums, put $S_F=\pi(F)$, and call block $t$ active if $0<s_t<m_t$ for some $s\in S_F$. Then
\[
\dim\conv F=\sum_{t\text{ active}}(m_t-1)+\dim\conv S_F.
\]
\end{proposition}
\begin{proof}
Put $L=\spn(F-F)$, and let $K_t$ be the zero-sum space supported on block $t$. In an active block, exchanging a zero and a one generates its coordinate differences, so $K_t\subseteq L$. Every vector of $L$ is constant on each inactive block. Therefore
\[
L\cap\ker\pi=\bigoplus_{t\text{ active}}K_t,
\qquad \pi(L)=\spn(S_F-S_F).
\]
Rank--nullity gives the formula.
\end{proof}

The map $x\mapsto(\one-x)/2$ identifies nonempty kernel-sign sets with the binary solutions of $Y\chi=Y\one/2$. Since the sign set is centrally symmetric, its affine and linear spans coincide, so $\sigma(Y)$ is the affine dimension of this binary polytope. Proposition~\ref{prop:binarydimension} therefore extends the multiplicity formula beyond the balanced right-hand side. In particular, it recovers Lee's full-dimensionality theorem for an equality-knapsack class \cite[Theorem~2.1(a)]{Lee97}.

Appendix~\ref{ec:lee} verifies Lee's special case.

For a totally unimodular integer matrix $Y$ with $\ker Y\ne\{0\}$, the kernel signs span the kernel exactly when $Y\one$ is even in every coordinate. Remark~\ref{prop:tu} proves this integrality consequence and gives the dimension formula, including the zero-kernel case.

\subsection{Equality-knapsack specialization\label{ec:lee}}
\begin{remark}\label{rem:knapsack}
Proposition~\ref{prop:binarydimension} recovers the following result of Lee \cite[Theorem~2.1(a)]{Lee97}. Let $b\ge2$ be an integer, let the sets $R_1,\ldots,R_b$ partition $\{1,\ldots,n\}$, allowing empty sets, and suppose that $|R_1|\ge b+1$. Then
\[
P_b=\conv\left\{x\in\{0,1\}^n:
\sum_{i=1}^b i\sum_{j\in R_i}x_j=b\right\}
\]
has dimension $n-1$.

To verify this specialization, write $m_i=|R_i|$, let $I=\{i:m_i>0\}$, and set $r=|I|$. The count patterns
\[
be_1,\qquad (b-i)e_1+e_i\quad(i\in I\setminus\{1\})
\]
are feasible and affinely independent. All count patterns satisfy $\sum_{i\in I}is_i=b$, so their convex hull has dimension $r-1$. The first pattern makes block $1$ active, since $0<b<m_1$. For every other block of size at least two, the corresponding pattern has the interior count $s_i=1$. Thus all blocks of size at least two are active. Proposition~\ref{prop:binarydimension} gives
\[
\dim P_b=\sum_{i\in I}(m_i-1)+(r-1)=(n-r)+(r-1)=n-1.
\]
This calculation uses precisely the hypotheses stated above; it does not assert full dimensionality for general equality-knapsack polytopes.

\end{remark}

\subsection{Totally unimodular factors}
\begin{remark}\label{prop:tu}
For a totally unimodular matrix $Y\in\mathbb Z^{d\times n}$, the binary interpretation gives a second tractable case:
\[
\sigma(Y)=
\begin{cases}n-\rank Y,&Y\one\in(2\mathbb Z)^d,\\0,&\text{otherwise}.\end{cases}
\]
Indeed, a nonintegral right-hand side $Y\one/2$ excludes binary solutions. If it is integral, total unimodularity makes every vertex of $\{z\in[0,1]^n:Yz=Y\one/2\}$ binary. Since $\one/2$ satisfies all box inequalities strictly, this polytope has affine hull $\one/2+\ker Y$. The transformation $x=\one-2z$ proves the claim.
\end{remark}

\section{Replication thresholds and realizations}
\subsection{A quantitative replication threshold}\label{ec:threshold}
For a rational kernel, the uniform threshold $\kappa(B)$ depends only on the subspace $\ker B$. The proof of Corollary~\ref{cor:rational} bounds it by the largest coordinate needed in an integer spanning set. Such a set can be constructed from maximal minors of any full-row-rank integer matrix defining the same kernel.

\begin{proposition}\label{prop:replicationbound}
Let $A\in\mathbb Z^{s\times r}$ have full row rank and satisfy $\ker A=\ker B$. If $0<s<r$, let $\Delta(A)$ be the largest absolute value of an $s\times s$ minor of $A$. Then
\[
\kappa(B)\le\Delta(A).
\]
If $\ker B=\{0\}$ or $\ker B=\mathbb R^r$, then $\kappa(B)=1$. In particular, if $B$ is totally unimodular, then $\kappa(B)=1$.
\end{proposition}

\begin{proof}
Suppose $0<s<r$, and choose a nonsingular $s$-column submatrix $A_I$. For each $j\notin I$, define $h^{(j)}\in\mathbb Z^r$ by
\[
h_j^{(j)}=\det A_I,
\qquad
h_I^{(j)}=-\operatorname{adj}(A_I)A_j,
\]
with all remaining coordinates zero. Then $Ah^{(j)}=0$. Each coordinate is, up to sign, an $s\times s$ minor of $A$, so
\[
\|h^{(j)}\|_\infty\le\Delta(A).
\]
The vectors $h^{(j)}$, $j\notin I$, are independent because their coordinates outside $I$ have disjoint nonzero supports. They form a real basis of $\ker A$, proving the bound.

The zero subspace needs no generators, while $\mathbb R^r$ is spanned by its coordinate vectors, so both extreme cases have threshold one. If $B$ is totally unimodular, select independent rows to obtain a full-row-rank matrix $A$ with the same kernel. In the nontrivial case, $\Delta(A)=1$, and the bound applies.
\end{proof}

\subsection{Exactness with an irrational kernel}\label{ec:irrational}
Every rational matrix has a rational kernel, so Theorem~\ref{thm:realization} applies in particular to every rational doubly nonnegative matrix with positive diagonal. The proof only needs rationality of the kernel, however, and uses that assumption only for full spanning. The following example shows why exactness alone does not require it.

\begin{remark}\label{rem:irrational}
Consider
\[
Z=
\begin{pmatrix}
1&\sqrt2\\
\sqrt2&2
\end{pmatrix}
=B^\top B,
\qquad
B=(1,\sqrt2).
\]
The matrix $Z$ is doubly nonnegative of rank one with positive diagonal. Its kernel is spanned by $(\sqrt2,-1)$ and contains no nonzero integer vector. Therefore $L_k(B)=\{0\}$ for every $k\in\mathbb Z_{>0}^2$.

For uniform replication, $Y=B^{(2k\one)}$ has $n=4k$ columns, and Theorem~\ref{thm:replication} gives
\[
\sigma(Y)=n-2<n-1=\dim\ker Y.
\]
Balanced signs within each block still certify exactness of the associated Max-Cut instance. Thus every positive even replication gives an exact instance, but none gives a sign-spanned dual kernel. Corollary~\ref{cor:rational} rules out full spanning under any positive replication of this factor, and Corollary~\ref{cor:deficit} accounts for the persistent deficit of one under even replication.
\end{remark}\subsection{Connected realizations at every rank}\label{ec:connected}
\begin{proof}[Proof of Corollary~\ref{cor:nonormalform}]
Let $Z_0$ be the rational rank-three matrix in Example~\ref{ex:horn14}, and let $H=J_5-2A(C_5)$ be its copositive Horn witness. The example verifies that $Z_0$ is doubly nonnegative with positive diagonal and connected support, and that $\langle H,Z_0\rangle=-2$. For $d=3$, apply Theorem~\ref{thm:realization} directly to $Z_0$.

For $d>3$, a direct sum with $I_{d-3}$ would raise the rank but leave the support disconnected. We add a rank-preserving nonnegative update to connect the components. Put
\[
A_d=Z_0\oplus I_{d-3},\qquad
w=Z_0e_1=(2,1,0,0,1)^\top,\qquad
u=\begin{pmatrix}w\\\one_{d-3}\end{pmatrix},
\]
and define $\widetilde Z_d=A_d+uu^\top$. This matrix is rational and doubly nonnegative with positive diagonal. Since $u=A_d(e_1^\top,\one_{d-3}^\top)^\top$ belongs to $\range A_d$, positivity gives
\[
\ker\widetilde Z_d=\ker A_d\cap u^\perp=\ker A_d,
\qquad \rank\widetilde Z_d=d.
\]
The update joins every new index to indices $1,2,5$ of the original connected support, so the support of $\widetilde Z_d$ is connected. Moreover,
\[
Hw=(0,0,2,2,0)^\top,\qquad w^\top Hw=0,
\]
and therefore
\[
\langle H\oplus0,\widetilde Z_d\rangle
=\langle H,Z_0\rangle+w^\top Hw=-2.
\]
The matrix $H\oplus0$ is copositive, so $\widetilde Z_d$ is not completely positive. Theorem~\ref{thm:realization} now gives a connected sign-spanned exact instance of dual rank $d$.
\end{proof}

\section{The fourteen-vertex example and its restricted minimum}
\subsection{Further calculations for the fourteen-vertex example}\label{ec:fourteen}
We use the factor $B$, multiplicities $m$, and block sums $a,b$ from Example~\ref{ex:horn14}.
The instance is exact and sign-spanned, with optimal slack $Y^\top Y$ of rank three. Its principal submatrix $Z_0$ excludes complete positivity. For a block sum $c$, the number of sign assignments is $\prod_i\binom{m_i}{(m_i+c_i)/2}$, giving $48,48,8,8$ for $a,-a,b,-b$ and hence $56$ optimal cuts up to sign. The dual objective is the sum of all entries of $Y^\top Y$, so the optimum is $\|Y\one\|^2=\|Bm\|^2=212$.

The block-sum filtering in Example~\ref{ex:horn14} is particularly short: the fourth coordinate forces $v\in\{-1,1\}$, and the fifth forces $t\in\{-2,0,2\}$. Of these six pairs, $(1,-2),(-1,2),(-1,0),(1,0)$ give $a,-a,b,-b$, respectively; the other two violate the coordinate bounds. Also $Bm=(0,4,14)^\top$, checking the objective directly. More generally,
\[
\cut_W(x)=\|Y\one\|^2-\|Yx\|^2.
\]
For uniform even replication, substitution in the same kernel parametrization gives
\[
\ker B\cap\{-1,0,1\}^5=\{0,\pm(0,1,-1,1,-1)\}.
\]
These vectors span only a line, whereas $a$ is independent of that line and has infinity norm two. Thus $\kappa(B)=2$: uniform even replication first gives full spanning on twenty vertices, while the displayed nonuniform replication uses fourteen.

The minimum order among positive replications of this particular factor is fourteen; Appendix~\ref{ec:minimum} gives the complete finite check.

\subsection{Verification of the restricted replication minimum\label{ec:minimum}}
For the five-column factor $B$ of Example~\ref{ex:horn14}, the minimum order of a sign-spanned positive replication is fourteen. This is a finite integer check: enumerate all $\binom{14}{5}=2002$ positive ordered multiplicity vectors of total size at most fourteen. For each, enumerate $|v|\le m_4$, $|t|\le m_5$, retain the block sums satisfying the bounds and parities, and test activity of every block of size at least two and rank two of the retained parameter pairs. Theorem~\ref{thm:multiplicity} proves this test complete. The only successful vectors are
\[
(4,2,1,3,4),\qquad(4,3,2,2,3),\qquad(4,4,3,1,2).
\]
For example, $(4,4,2,1,2)$ fails because $v$ must be odd and $t$ even, making $c_3=-3v-2t$ odd although its required parity is even. This restricted minimum makes no claim about other rank-three non-completely-positive constructions.

\subsection{Executable verification of the replication minimum}\label{ec:code}
This Python~3 code implements the bound, parity, activity, and rank tests of Appendix~\ref{ec:minimum}. It uses only built-in operations, requires no packages or input files, and can be copied into a file and run directly. Each positive ordered composition is tested once; $|v|\le m_4$ and $|t|\le m_5$ ensure that every attainable block sum is examined.

\begin{Verbatim}[fontsize=\small]
def kernel_sums(m):
    out = []
    for v in range(-m[3], m[3]+1):
        for t in range(-m[4], m[4]+1):
            c = (-2*v-2*t, 4*v+3*t, -3*v-2*t, v, t)
            if all(abs(a) <= b and (a-b) % 2 == 0
                   for a, b in zip(c, m)):
                out.append(c)
    return out

def full_span(m):
    C = kernel_sums(m)
    active = all(b == 1 or any(abs(c[i]) < b for c in C)
                 for i, b in enumerate(m))
    rank_two = any(c[3]*d[4] != c[4]*d[3]
                   for c in C for d in C)
    return active and rank_two

def compositions(n, r):
    if r == 1:
        yield (n,)
    else:
        for a in range(1, n-r+2):
            for tail in compositions(n-a, r-1):
                yield (a,) + tail

success = []
for n in range(5, 15):
    candidates = list(compositions(n, 5))
    good = [m for m in candidates if full_span(m)]
    print(n, len(candidates), good)
    success.extend(good)
assert success == [(4,2,1,3,4), (4,3,2,2,3), (4,4,3,1,2)]
\end{Verbatim}

There are $\sum_{n=5}^{14}\binom{n-1}{4}=\binom{14}{5}=2002$ candidates. The expected numbers tested at totals $5,6,\ldots,14$ are
\[
1,\ 5,\ 15,\ 35,\ 70,\ 126,\ 210,\ 330,\ 495,\ 715.
\]
All totals below fourteen give an empty output list. At fourteen the successful vectors are exactly
\[
(4,2,1,3,4),\qquad (4,3,2,2,3),\qquad (4,4,3,1,2).
\]
Thus fourteen is the minimum for positive column replications of this particular factor. This computation makes no claim about the minimum order of a non-completely-positive sign-spanned exact instance obtained by other constructions.

\section{Low-rank normal forms and excess cp-rank}\label{ec:lowrank}
\label{sec:lowrank}\label{sec:gram}
The realization theorem preserves nonnegative factorization properties rather than imposing them. At ranks one and two these properties are automatic, which explains the contrast with the preceding example.

\begin{proposition}\label{ec:rankone}\label{thm:ranktwo}
An exact instance has dual rank one precisely when
\[
W_{ij}=4v_iv_j\quad(i\ne j),\qquad v>0,\qquad v^\perp\cap\sgn^n\ne\varnothing.
\]
It has dual rank two precisely when $W_{ij}=4(u_iu_j+w_iw_j)$ off the diagonal for independent nonnegative rows $u,w$, with no zero column and with a sign vector in the kernel of the two-row factor. In both cases $c(W)=\sigma(Y)$, where $Y=v^\top$ in rank one and $Y$ has rows $u,w$ in rank two.
\end{proposition}
\begin{proof}
A rank-one doubly nonnegative slack with positive diagonal is $vv^\top$ with $v>0$. For rank two, the nonzero Gram columns have pairwise nonnegative inner products. Fixing one direction at angle zero puts all others in $[-\pi/2,\pi/2]$. The difference between the largest and smallest angles lies in $[0,\pi]$ and has nonnegative cosine, so it is at most $\pi/2$. Thus all directions lie in a quadrant, and a rotation makes both coordinates nonnegative \cite[Remark~1.1]{BrandtsKrizek16}. Conversely, each prescribed factor gives a complementary doubly nonnegative certificate by Proposition~\ref{prop:dnn}. Equation~\eqref{eq:cuts} identifies the optimal signs.
\end{proof}

The rank-two form also identifies the support: if $S_0=\{i:u_i=0\}$ and $T_0=\{i:w_i=0\}$, the only nonedges join $S_0$ to $T_0$. These sets are disjoint because no column is zero. This follows directly from the vanishing of the two nonnegative summands in $u_iu_j+w_iw_j$. A kernel sign vector already rules out isolated vertices: if $y_i$ were orthogonal to all other columns, taking the inner product of $Yx=0$ with $y_i$ would give $x_i\|y_i\|^2=0$.

Thus maximal cut-span $c(W)=n-1$ forces rank-one product weights, and Theorem~\ref{thm:multiplicity} decides when their signs span $v^\perp$. Exact product-weight instances need not have dual rank one: the Laurent--Poljak exactness criterion also includes the case where one coordinate exceeds the sum of the others \cite[Theorem~3.3]{LP95}. The requirement $v^\perp\cap\sgn^n\ne\varnothing$ excludes that regime.

Complete positivity at higher rank does not guarantee a nonnegative factor of minimum dimension either. Let
\[
N=\frac1{\sqrt2}\begin{pmatrix}1&1&0&0\\0&1&1&0\\0&0&1&1\\1&0&0&1\end{pmatrix},
\qquad Z=N^\top N=I+\tfrac12A(C_4).
\]
The matrix $Z$ has rank three and kernel spanned by $(1,-1,1,-1)^\top$, making it a sign-spanned optimal slack for a four-cycle with weights two. Every row of a nonnegative factor is supported on a clique of $C_4$. Each such clique covers at most one edge, so at least four rows are needed; the displayed factor proves cp-rank four. Replication preserves this excess, just as it preserves failure of complete positivity.

\section{Rational construction of the square-root-sum family\label{ec:arithmetic}}
This section retains a linear-size arithmetic illustration of the general equivalence. Write $\mathrm{SRS}_{\le}$ for deciding $\sum_i\sqrt{a_i}\le b$ on positive binary integers.

\begin{theorem}[Square-root-sum reduction]\label{thm:arithmetic}
Given positive binary integers $a_1,\ldots,a_m,b$, one can construct in deterministic polynomial time an exact rational positively weighted complete graph on $n=16m+20$ vertices, its unique optimal cut $xx^\top$, and a rational nonnegative matrix $N$ whose Gram matrix is the optimal slack $Z^*$. Writing $S=\sum_i\sqrt{a_i}$, the construction satisfies
\[
F^*(W)=P^*(W)=\{xx^\top\}\quad\Longleftrightarrow\quad S<b.
\]
A strictly complementary optimal pair exists precisely when $S>b$. All output data have polynomial bit length. The weights lie in $[4(n+1),4(n+3)]$, and the optimum is $n(n+1)^2$.
\end{theorem}

\subsection{A projector normalization of the embedding}
For this family we use the same matrices $A,C,U,x$ and face description as in Lemma~\ref{lem:embedding}, with a different certificate normalization. For arbitrary full-column-rank rational $B$, put
\[
P=I-U(U^\top U)^{-1}U^\top,\qquad N=P+J_n,\qquad
Z=N^\top N=P+(n+2)J_n.
\]
Since $U^\top\one_n=0$, $PJ_n=J_nP=J_n$ and $\ker Z=\range U$. The projector has entries in $[-1,1]$, so $N\ge0$ and $Z_{ij}\in[n+1,n+3]$ for $i\ne j$. Setting $W_{ij}=4Z_{ij}$ off the diagonal gives weights in $[4(n+1),4(n+3)]$. Also $Z\one_n=(n+1)^2\one_n$, so $Z=\Diag((n+1)^2\one_n)-L_W/4$ and the optimum is $n(n+1)^2$. Complementarity and the proof of~\eqref{eq:pyramid} give exactly the same optimal face and unique cut as before. Rational inversion gives polynomial bit length; this normalization makes no polynomial-magnitude claim after clearing denominators.

\subsection{Encoding the arithmetic comparison}

The arithmetic comparison has a simple Gram interpretation. Take two vectors of each length $\sqrt{a_i}$. Their sum can have any length between zero and $2\sum_i\sqrt{a_i}$: opposing each pair gives zero, while alignment gives the maximum. Requiring length $2b$ therefore encodes the desired comparison. The squared lengths and squared target are integers; the next lemma expresses these constraints using rational rows and measurements equal to one.

\begin{lemma}\label{lem:roots}
For positive integers $a_1,\ldots,a_m,b$, one can construct a full-column-rank rational matrix $B$ with $4m+3$ rows and $2m+1$ columns such that
\[
\mathcal T(B)\ne\varnothing\quad\Longleftrightarrow\quad\sum_{i=1}^m\sqrt{a_i}\ge b.
\]
Furthermore, $\mathcal T(B)$ contains a positive definite matrix if and only if the inequality is strict.
\end{lemma}
\begin{proof}
Index the coordinates by $0$ and $(i,\varepsilon)$, where $1\le i\le m$ and $\varepsilon\in\{+,-\}$. Let $e_0$ be the first coordinate vector, let $f_{i,\varepsilon}$ denote the remaining coordinate vectors, and put $z=\sum_{i,\varepsilon}f_{i,\varepsilon}$. For a positive rational $c$, set
\[
h_c=\frac{c-1}{c+1},\qquad g_c=\frac{2}{c+1},\qquad h_c^2+c g_c^2=1.
\]
The rows of $B$ are
\[
e_0^\top;\qquad
(h_{a_i}e_0\pm g_{a_i}f_{i,\varepsilon})^\top\quad(i,\varepsilon);
\qquad
(h_{4b^2}e_0\pm g_{4b^2}z)^\top.
\]
There are $1+4m+2$ rows. Their differences, together with $e_0$, span every coordinate direction, so $B$ has full column rank. The entries have polynomial bit length in the original integers.

For $R\in\mathcal T(B)$, the first measurement gives $R_{00}=1$. Adding the two measurements in each pair shows that the principal block $Q$ on the other $2m$ coordinates satisfies
\begin{equation}\label{eq:gramcondition}
Q\succeq0,\qquad Q_{(i,\varepsilon),(i,\varepsilon)}=a_i,
\qquad \one^\top Q\one=4b^2.
\end{equation}
Conversely, any $Q$ satisfying~\eqref{eq:gramcondition} gives $[1]\oplus Q\in\mathcal T(B)$. The possible additional equations on the first row cause no difficulty, since this block-diagonal choice satisfies all of them.

Represent $Q$ as a Gram matrix of vectors $v_{i,+},v_{i,-}$ of norm $\sqrt{a_i}$. The last equality in~\eqref{eq:gramcondition} says
\[
\left\|\sum_{i=1}^m(v_{i,+}+v_{i,-})\right\|=2b.
\]
The triangle inequality gives $b\le S:=\sum_i\sqrt{a_i}$. Conversely, if $0<b\le S$, take orthonormal vectors $u,w$, put $\theta=b/S$, and choose
\[
v_{i,\pm}=\sqrt{a_i}\bigl(\theta u\pm\sqrt{1-\theta^2}\,w\bigr).
\]
Their norms and sum give~\eqref{eq:gramcondition}. These algebraic vectors are used only to establish existence; they are not computed or included in the reduction.

\smallskip\noindent\emph{Positive definite feasibility.} To distinguish the strict comparison from its boundary, let $D=\Diag(a_1,a_1,\ldots,a_m,a_m)$ and let $Q_-$ be block diagonal with blocks $a_i\bigl(\begin{smallmatrix}1&-1\\-1&1\end{smallmatrix}\bigr)$. Finally let $Q_+=ww^\top$, where here $w=(\sqrt{a_1},\sqrt{a_1},\ldots,\sqrt{a_m},\sqrt{a_m})^\top$. All three matrices have the required diagonal, and
\[
\one^\top Q_-\one=0,\qquad \one^\top D\one=2\sum_i a_i,
\qquad \one^\top Q_+\one=4S^2.
\]
Every number strictly between $0$ and $4S^2$ is attained by $D$ itself or by a convex combination of $D$ with $Q_-$ or $Q_+$ in which $D$ has positive coefficient. That combination is positive definite. Thus $b<S$ gives a positive definite $Q$ in~\eqref{eq:gramcondition} and a positive definite member of $\mathcal T(B)$. If $b=S$, equality in the triangle inequality forces all $2m$ Gram vectors to have the same direction. Every feasible $Q$ then has rank one, so no feasible $R$ is positive definite. If $b>S$, the set is empty.
\end{proof}

\begin{proof}[Proof of Theorem~\ref{thm:arithmetic}]
Apply the projector normalization above to the matrix from Lemma~\ref{lem:roots}. Its order is $n=4((4m+3)+2)=16m+20$. The equality criterion and uniqueness follow from Lemma~\ref{lem:embedding}; the bit, weight, and objective bounds follow from the projector normalization.

For strict complementarity, the slack has kernel $\range U$ of dimension $2m+3$. A primal optimum has full rank on this kernel precisely when its coefficient matrix $R$ is positive definite. In~\eqref{eq:pyramid}, choosing $0<\tau<1$ and $H\succ0$ makes both diagonal blocks positive definite. Conversely, a positive definite coefficient matrix requires $0<\tau<1$ and $H\succ0$. Thus this happens if and only if $\mathcal T(A)$ contains a positive definite matrix. The gadget forming $A$ preserves positive definite feasibility: restriction gives one direction, and $Q\mapsto Q\oplus[1]$ gives the other. Lemma~\ref{lem:roots} therefore gives strict complementarity precisely when $S>b$.
\end{proof}

\subsection{Consequences for exact recognition}
\begin{corollary}\label{cor:srs}
There is a polynomial-time many-one reduction $\mathrm{SRS}_{\le}\le_p\mathsf{CH}$, even for the restricted inputs in Theorem~\ref{thm:arithmetic}. Consequently, $\mathsf{CH}\in\mathsf P$ would imply $\mathrm{SRS}_{\le}\in\mathsf P$, and $\mathsf{CH}\in\mathsf{NP}$ would imply $\mathrm{SRS}_{\le}\in\mathsf{NP}$.
\end{corollary}
\begin{proof}
If every $a_i$ is a perfect square, integer square roots decide the comparison, and we output a fixed yes- or no-instance. For example, the construction with $(a_1,b)=(1,2)$ gives yes and that with $(a_1,b)=(2,1)$ gives no. Otherwise $S$ is irrational: grouping the radicals by square-free parts expresses it as a positive rational combination of distinct square-free radicals, which are linearly independent together with $1$ \cite{EHS24}. Hence $S\le b$ is equivalent to $S<b$, and Theorem~\ref{thm:arithmetic} applies. Only perfect-square tests are used by the algorithm; no integer factorization is needed.
\end{proof}

Square-root-sum comparison is a classical obstacle in exact semidefinite computation \cite{Goemans98,EHS24}. Tarasov and Vyalyi reduce arithmetic-circuit comparisons to exact SDP feasibility \cite{TarasovVyalyi08}; Allender et al. relate square-root sums to PosSLP and place both problems in the counting hierarchy \cite{AllenderEtAl09}. These results place the explicit family here within the broader arithmetic-complexity setting. Here the comparison orientation is explicitly $S\le b$, since both orientations appear in the literature. The explicit reduction illustrates the general exact-feasibility equivalence proved in Section~\ref{sec:sdp-complexity}; it also records a linear vertex count and the exact strict-complementarity boundary. It is not an $\mathsf{NP}$-hardness result. It is consistent with Corollary~\ref{cor:strict}: the constructed equality instances have a one-dimensional effective support inside a $(2m+3)$-dimensional dual kernel and therefore fail strict complementarity.

\section{Fitting with arbitrary rational coefficients}
\begin{remark}\label{rem:rational-fitting}
The affine-lift conclusion of Lemma~\ref{lem:fitting} also holds for an arbitrary rational system~\eqref{eq:bounded}, with running time polynomial in its binary encoding length. The coefficient and common-denominator magnitude bounds are no longer asserted. To see this, polarize by~\eqref{eq:polar} and express each positive rational coefficient $a/b$ as a sum of polynomially many rational squares: a binary term $2^{2j}$ in $ab$ is $(2^j)^2$, and $2^{2j+1}$ is the sum of two such squares; divide the square roots by $b$. This uses no factorization. Absorb these scalars into the forms in~\eqref{eq:sum-norms}. If $T$ is the sum of their squared coefficient norms, scale all forms by $1/(1+T)$. Their combined squared norm is at most $1/4$ on $Q\succeq0$, $\tr Q\le1$, so the copy and addition gadgets apply. Every scale has polynomial bit length. The proof of Lemma~\ref{lem:fitting} now gives feasibility in both directions, a rational affine lift, and full column rank. The controlled integer expansion in the main text is what additionally allows polynomial-size unweighting.
\end{remark}

\section{Explicit parameters for the sampling bound\label{ec:radius}}
\begin{remark}[Explicit sampling parameters]\label{rem:radius}
For completeness, let $k$ be the number of variables of the nonconstant polynomial $f$, $d\le4$ its degree, and $\tau$ a coefficient bit bound. In the notation of \cite[Theorem~2]{BasuRoy10}, put
\[
d'=\max\{2(d+1),6\},\quad D_0=k(d'-2)+2,\quad N=d'(d'-1)^{k-1},
\]
\begin{align*}
\tau_0&=2\tau+k\bit(d+1)+\bit(2d'),\\
\tau_1&=D_0\bigl(\tau_0+4\bit(2D_0+1)+\bit N\bigr)
             -2\bit(2D_0+1)-\bit N,\\
\tau_2&=\tau_1+2(k-1)\bit N+(2k-1)\bit k,\\
\tau'&=N\bigl(\tau_2+\bit N+2\bit(2D_0+1)+1\bigr),\\
C&=2D_0N(2N-1)+1,\\
E&=(2N-1)\bigl(\tau'+\bit(2N-1)+\bit(2D_0N+1)\bigr).
\end{align*}
The squared sampling-radius bound is $C2^E$. Since $d\le4$, $\bit N=O(k)$; all integers displayed here have bit length polynomial in $k+\tau$. This supplies an explicit polynomial-time choice of $h$ in the proof. This use of a real algebraic sample bound assumes neither a rational sample nor a short rational feasible solution.
\end{remark}

\section{Real certificates for both recognition outcomes\label{ec:real}}
\begin{proof}[Detailed proof of Corollary~\ref{prop:c1-real-complexity}]
We encode a certified computation of the oracle algorithm in Theorem~\ref{thm:oracle}. There are two issues: certifying either answer to each SDP query, and ensuring that the answers describe one consistent adaptive computation.

For a rational SDP $D$, feasibility has a polynomial-size system of real polynomial equations: write each PSD matrix as $VV^\top$ and substitute into the affine constraints. Rational denominators can be cleared. Apply the same construction to Ramana's exact alternative for $D$ to obtain a polynomial-size system certifying infeasibility of $D$. Thus there are uniformly constructible systems
\[
f_D(u)=0\quad\Longleftrightarrow\quad D\text{ is feasible},
\qquad
g_D(v)=0\quad\Longleftrightarrow\quad D\text{ is infeasible},
\]
where the equivalences mean existence of the displayed real witnesses. The second equivalence includes weak infeasibility.

Fix the input and guess the oracle-answer bits $a_1,\ldots,a_T$, padding the computation to a polynomial bound on its number of queries. After termination, use a fixed feasible dummy query. A Boolean circuit simulates the deterministic algorithm with these answers and produces each query $D_j$ from the preceding answers. We may impose a polynomial clock and reject malformed intermediate encodings; these conventions do not change the actual computation. For each query impose, componentwise,
\[
a_j(a_j-1)=0,\qquad
 a_j f_{D_j}(u_j)=0,\qquad
 (1-a_j)g_{D_j}(v_j)=0.
\]
Thus only the certificate selected by the answer bit is required. The circuit also performs the polynomial-time construction of each alternative system.

All circuit wires are constrained to be Boolean. For example, an AND gate is encoded by $z=xy$ and a NOT gate by $z=1-x$, together with Boolean equations for the wires. The output bits specify signed integer numerators and positive integer denominators of the SDP coefficients. Reconstruct those integers by their binary expansions and introduce a real coefficient $c$ with $dc=p$ for a numerator $p$ and denominator $d\ge1$. The selected Gram equations therefore use exactly the encoded rational coefficients. Query dimensions, numbers of equations, and witness matrices can all be padded to polynomial bounds. Keeping one variable per gate or intermediate expression avoids expanding compositions, so the total number of variables and the encoding length of the polynomial equations remain polynomial in the input length.

Finally require the simulated output to be yes. Any solution certifies the first oracle answer correctly; induction then certifies each subsequent query and answer, since its query is determined by the preceding transcript. Hence the simulation is the actual oracle computation and its yes answer is correct. Conversely, the actual yes computation has witnesses for every selected feasibility or infeasibility certificate and satisfies the whole system. This proves membership in $\exists\mathbb R$. Requiring the final output to be no gives the same argument for the complement on the promise domain. The real witnesses need not have short rational encodings.
\end{proof}

Strict containment also has a direct geometric certificate. It holds precisely when the optimal face has an extreme matrix $X=V^\top V$ of rank $r\ge2$. Indeed, this compact face is the convex hull of its extreme points; if they were all rank one, they would all be cut matrices. Conversely, a higher-rank extreme point cannot be a convex combination of cut matrices. Since an extreme point of a face is extreme in the elliptope itself, the elliptope extremality test applies. The unit-column conditions and $VZ^*=0$ give feasibility. By Li--Tam extremality \cite{LiTam94}, the matrices $v_iv_i^\top$ must span $\mathbb S^r$. If $E(V)$ records their upper-triangular entries as columns, a right inverse $E(V)D=I_{r(r+1)/2}$ certifies this condition by polynomial equations. The possible ranks satisfy $r(r+1)/2\le n$. A higher-rank optimum alone would not suffice, as the $K_4$ triangle illustrates; extremality is the essential additional condition.

\section{Orthogonality and finite classifications}
\subsection{Orthogonality numbers by order\label{ec:orthogonal}}
For completeness, let $\kappa_\perp(n)$ be the maximum of $o(W)$ over the instances of order $n$ under consideration. Parity gives $\kappa_\perp(n)=1$ for odd $n$, and $\kappa_\perp(2)=1$. If $n\equiv2\pmod4$, three orthogonal signs cannot exist: switching the first to $\one$ makes the other two balanced, and their inner product is $4a-n$, where $a$ counts their common positive coordinates. Thus at most two are possible.

\begin{proposition}\label{prop:orthmaximum}
For every $n\equiv2\pmod4$ with $n\ge6$, $\kappa_\perp(n)=2$. At every Hadamard order, $\kappa_\perp(n)=n-1$.
\end{proposition}
\begin{proof}
For $n=6$, take a unit-weight $K_4$ on $\{1,2,3,4\}$ and add pendant edges $\{1,6\}$ and $\{2,5\}$. Its SDP value is at most $4+1+1=6$, attained by the orthogonal cuts
\[
(1,1,-1,-1,-1,-1),\qquad(1,-1,1,-1,1,-1).
\]
To add four vertices, attach two leaves where the cuts agree and two where they disagree, assigning every leaf the opposite sign to its parent. The new inner-product contributions cancel, and every added edge attains its individual upper bound. Exactness and orthogonality persist. The parity bound proves optimality; Theorem~\ref{thm:hadamard} treats Hadamard orders.
\end{proof}

For $K_6$ itself, balanced signs have pairwise inner product congruent to two modulo four. Hence $c(K_6)=5$ but $o(K_6)=1$: maximal span and orthogonality are distinct even among unweighted graphs.

\subsection{Exact finite classifications}
\label{sec:computations}
For the direction search, enumerate nondecreasing positive integer vectors within the bound of Proposition~\ref{prop:cofactor}, discard nonprimitive vectors and vectors with odd coordinate sum, and enumerate their kernel signs with first coordinate $+1$. An exact rational rank test decides whether these signs span a hyperplane. The parity filter is necessary because $p^\top x\equiv\sum_i p_i\pmod2$. This test is independent of the multiplicity formula. The classical maximal-determinant values $D_{01}(j)=1,1,1,2,3,5,9$ for $0\le j\le6$ \cite{BrentOsborn13} give the sufficient bounds $b_n=D_{01}(n-2)$ below.

\begin{table}[ht]
\centering
\begin{tabular}{@{}ccccc@{}}
\toprule
$n$ & \small Hadamard bound & $b_n$ & \small Directions & \small Largest coordinate\\
\midrule
2 & 1.00 & 1 & 1 & 1\\
3 & 1.00 & 1 & 0 & --\\
4 & 1.30 & 1 & 1 & 1\\
5 & 2.00 & 2 & 1 & 2\\
6 & 3.49 & 3 & 4 & 3\\
7 & 6.75 & 5 & 14 & 5\\
8 & 14.18 & 9 & 122 & 9\\
\bottomrule
\end{tabular}
\caption{\small Complete enumeration of primitive positive integer directions $p$ with $\sigma(p^\top)=n-1$, up to coordinate permutation. The second column displays the general Hadamard upper bound rounded to two decimals; $b_n$ is the sharper sufficient coordinate bound obtained from maximal determinants. Every nonempty row attains $b_n$.}
\label{tab:directions}
\end{table}

At $n=6$ the four directions are
\[
(1,1,1,1,1,1),\quad(1,1,1,1,1,3),\quad
(1,1,1,1,2,2),\quad(1,1,1,2,2,3).
\]
At $n=8$, the directions $(1,2,3,4,5,6,7,8)$ and $(1,1,2,2,3,4,4,9)$ illustrate the range; the latter attains the determinant bound nine. The enumeration can use the larger integer Hadamard bound without changing the resulting list.

A separate enumeration considers connected labelled simple graphs through order six. For each graph, enumerate maximum cuts with first coordinate $+1$ and test for an orthogonal pair. When such a pair exists, choose one maximum cut $x$ and test
\[
4Z_x=W+\Diag(-x\circ(Wx))
\]
for positive semidefiniteness using rational Schur complements. A negative diagonal rejects; a zero diagonal requires the corresponding row to vanish; a positive diagonal permits elimination and recursion. Proposition~\ref{prop:certificate} proves that this is an exactness test. Minimizing the adjacency encoding over all vertex permutations identifies the successful graphs up to isomorphism. No SDP solver or numerical eigenvalue tolerance is involved.

\begin{table}[ht]
\centering
\begin{tabular}{@{}c c c c@{}}
\toprule
$n$ & \small Connected labelled graphs & \small Exact with $o(W)\ge2$ & \small Isomorphism classes\\
\midrule
2&1&0&0\\
3&4&0&0\\
4&38&1&1\\
5&728&0&0\\
6&26704&180&1\\
\bottomrule
\end{tabular}
\caption{\centering \small Certified enumeration of simple unweighted graphs. The second and third columns count labelled graphs; the fourth counts isomorphism classes.}
\label{tab:census}
\end{table}

The four-vertex class is $K_4$. The six-vertex class is precisely the pendant-edge graph in Proposition~\ref{prop:orthmaximum}. Its automorphism group has order four, giving $6!/4=180$ labelled copies. These two classifications concern different objects: integer directions for weighted complete graphs, and connected simple unweighted graphs, respectively.
\FloatBarrier

\bibliographystyle{amsplain}
\bibliography{refs}

@book{de1997geometry,
  author    = {Deza, Michel M. and Laurent, Monique},
  title     = {Geometry of Cuts and Metrics},
  publisher = {Springer Science \& Business Media},
  volume    = {15},
  year      = {1997}
}

@article{DP93_maxcut,
  title={Laplacian eigenvalues and the maximum cut problem},
  author={Delorme, Charles and Poljak, Svatopluk},
  journal={Mathematical Programming},
  volume={62},
  number={1-3},
  pages={557--574},
  year={1993},
  publisher={Springer}
}

@article{goemans1995improved,
  title={Improved approximation algorithms for maximum cut and satisfiability problems using semidefinite programming},
  author={Goemans, Michel X. and Williamson, David P.},
  journal={Journal of the Association for Computing Machinery},
  volume={42},
  number={6},
  pages={1115--1145},
  year={1995},
  publisher={ACM}
}

@article{LP95,
  title={On a positive semidefinite relaxation of the cut polytope},
  author={Laurent, Monique and Poljak, Svatopluk},
  journal={Linear Algebra and its Applications},
  volume={223--224},
  pages={439--461},
  year={1995},
  publisher={Elsevier}
}

@misc{Bhardwaj2X,
  title={On exactness of {SDP} relaxation for the maximum cut problem},
  author={Bhardwaj, Avinash and Gogoi, Hritiz and Narayanan, Vishnu and Pathapati, Abhishek},
  year={2025},
  eprint={2505.05200},
  archivePrefix={arXiv},
  primaryClass={math.OC},
  howpublished={\href{https://arxiv.org/abs/2505.05200v8}{arXiv:2505.05200v8}},
  note={Version 8, revised June 29, 2026}
}

@article{MW2X,
  title={Max cut and semidefinite rank},
  author={Mirka, Renee and Williamson, David P.},
  journal={Operations Research Letters},
  volume={53},
  pages={107067},
  year={2024}
}

@book{berman2003completely,
  title={Completely Positive Matrices},
  author={Berman, Abraham and Shaked-Monderer, Naomi},
  year={2003},
  publisher={World Scientific}
}

@book{horadam2007hadamard,
  title={Hadamard Matrices and Their Applications},
  author={Horadam, Kathy J.},
  year={2007},
  publisher={Princeton University Press}
}

@article{DP93_complexity,
  title={Combinatorial properties and the complexity of a max-cut approximation},
  author={Delorme, Charles and Poljak, Svatopluk},
  journal={European Journal of Combinatorics},
  volume={14},
  number={4},
  pages={313--333},
  year={1993},
  publisher={Elsevier}
}

@misc{HLW21,
  title={Optimal solutions and ranks in the max-cut {SDP}},
  author={Hong, Daniel and Lee, Hyunwoo and Wei, Alex},
  year={2021},
  eprint={2109.02238},
  archivePrefix={arXiv},
  primaryClass={math.OC},
  howpublished={\href{https://arxiv.org/abs/2109.02238}{arXiv:2109.02238}}
}

@misc{BhardwajHardness,
  title={On the complexity of recognizing {SDP} exactness for the maximum cut problem},
  author={Bhardwaj, Avinash},
  year={2026},
  eprint={2609.03508},
  archivePrefix={arXiv},
  primaryClass={math.OC},
  howpublished={\href{https://arxiv.org/abs/2609.03508}{arXiv:2609.03508}}
}

@article{LP96,
  author={Laurent, Monique and Poljak, Svatopluk},
  title={On the facial structure of the set of correlation matrices},
  journal={SIAM Journal on Matrix Analysis and Applications},
  volume={17}, 
  number={3}, 
  pages={530--547}, 
  year={1996}
}

@article{deCarliSilvaTuncel19,
  author={de Carli Silva, Marcel K. and Tun{\c{c}}el, Levent},
  title={Strict complementarity in semidefinite optimization with elliptopes including the {MaxCut SDP}},
  journal={SIAM Journal on Optimization}, 
  volume={29}, 
  number={4},
  pages={2650--2676}, 
  year={2019}
}

@article{BrandtsKrizek16,
  author={Brandts, Jan and K{\v r}{\'i}{\v z}ek, Michal},
  title={Factorization of {CP}-rank-3 completely positive matrices},
  journal={Czechoslovak Mathematical Journal}, 
  volume={66}, 
  number={3},
  pages={955--970}, 
  year={2016}
}

@article{GoldbergerStrassler22,
  author={Goldberger, Assaf and Strassler, Yossi},
  title={A practical algorithm for completing half-{Hadamard} matrices using {LLL}},
  journal={Journal of Algebraic Combinatorics},
  volume={55},
  pages={217--244},
  year={2022},
  doi={10.1007/s10801-021-01077-z}
}

@article{AlonVu97,
  author={Alon, Noga and Vu, Van H.},
  title={Anti-{Hadamard} matrices, coin weighing, threshold gates, and indecomposable hypergraphs},
  journal={Journal of Combinatorial Theory, Series A},
  volume={79},
  number={1},
  pages={133--160},
  year={1997}
}

@article{RamanaGoldman95,
  author  = {Ramana, Motakuri and Goldman, Alan J.},
  title   = {Some geometric results in semidefinite programming},
  journal = {Journal of Global Optimization},
  volume  = {7},
  pages   = {33--50},
  year    = {1995},
  doi     = {10.1007/BF01100204}
}

@article{KuengTropp21,
 author = {Kueng, Richard and Tropp, Joel A.},
 title = {Binary Component Decomposition Part {I}: The Positive-Semidefinite Case},
 journal = {SIAM Journal on Mathematics of Data Science},
 volume = {3}, number = {2}, pages = {544--572}, year = {2021},
 doi = {10.1137/19M1278612}
}

@article{BrentOsborn13,
  author  = {Brent, Richard P. and Osborn, Judy-anne H.},
  title   = {On minors of maximal determinant matrices},
  journal = {Journal of Integer Sequences},
  volume  = {16},
  number  = {4},
  year    = {2013},
  pages   = {1--30},
  note    = {Article 13.4.2}
}

@incollection{Margot10,
  author = {Margot, Fran{\c c}ois},
  title = {Symmetry in Integer Linear Programming},
  booktitle = {50 Years of Integer Programming 1958--2008},
  publisher = {Springer},
  year = {2010},
  pages = {647--686},
  doi = {10.1007/978-3-540-68279-0_17}
}

@article{LP96Gap,
  author = {Laurent, Monique and Poljak, Svatopluk},
  title = {Gap inequalities for the cut polytope},
  journal = {European Journal of Combinatorics},
  volume = {17},
  number = {2--3},
  pages = {233--254},
  year = {1996}
}

@article{HerrRehnSchurmann13,
  author = {Herr, Katrin and Rehn, Thomas and Sch{\"u}rmann, Achill},
  title = {Exploiting symmetry in integer convex optimization using core points},
  journal = {Operations Research Letters},
  volume = {41},
  pages = {298--304},
  year = {2013},
  doi = {10.1016/j.orl.2013.02.007}
}

@article{Bulutoglu23,
title = {Finding the dimension of a non-empty orthogonal array polytope},
journal = {Discrete Optimization},
volume = {45},
pages = {100727},
year = {2022},
issn = {1572-5286},
author = {Dursun A. Bulutoglu}
}

@article{Lee97,
  author  = {Lee, Eva K.},
  title   = {On Facets of Knapsack Equality Polytopes},
  journal = {Journal of Optimization Theory and Applications},
  volume  = {94},
  number  = {1},
  pages   = {223--239},
  year    = {1997}
}

@misc{KT19II,
  author = {Kueng, Richard and Tropp, Joel A.},
  title = {Binary Component Decomposition Part {II}: The Asymmetric Case},
  year = {2019},
  howpublished = {\href{https://arxiv.org/abs/1907.13602}{arXiv:1907.13602}}
}

@article{LLL82,
  author = {Lenstra, Arjen K. and Lenstra, Jr., Hendrik W. and Lov{\'a}sz, L{\'a}szl{\'o}},
  title = {Factoring polynomials with rational coefficients},
  journal = {Mathematische Annalen},
  volume = {261},
  pages = {515--534},
  year = {1982}
}

@article{LiTam94,
  author = {Li, Chi-Kwong and Tam, Bit-Shun},
  title = {A Note on Extreme Correlation Matrices},
  journal = {SIAM Journal on Matrix Analysis and Applications},
  volume = {15}, number = {3}, pages = {903--908}, year = {1994},
  doi = {10.1137/S0895479892240683}
}

@article{Ramana97,
  author = {Ramana, Motakuri V.},
  title = {An exact duality theory for semidefinite programming and its complexity implications},
  journal = {Mathematical Programming},
  volume = {77}, pages = {129--162}, year = {1997},
  doi = {10.1007/BF02614433}
}

@inproceedings{Goemans98,
  author = {Goemans, Michel X.},
  title = {Semidefinite programming and combinatorial optimization},
  booktitle = {Proceedings of the International Congress of Mathematicians, Vol. III (Berlin, 1998)},
  series = {Documenta Mathematica, Extra Volume},
  pages = {657--666},
  year = {1998}
}

@inproceedings{EHS24,
  author = {Eisenbrand, Friedrich and Haeberle, Matthieu and Singer, Neta},
  title = {An Improved Bound on Sums of Square Roots via the Subspace Theorem},
  booktitle = {40th International Symposium on Computational Geometry (SoCG 2024)},
  series = {Leibniz International Proceedings in Informatics},
  volume = {293},
  pages = {54:1--54:8},
  publisher = {Schloss Dagstuhl--Leibniz-Zentrum f{\"u}r Informatik},
  year = {2024},
  doi = {10.4230/LIPIcs.SoCG.2024.54}
}

@article{BasuRoy10,
 author={Basu, Saugata and Roy, Marie-Fran{\c{c}}oise},
 title={Bounding the radii of balls meeting every connected component of semi-algebraic sets},
 journal={Journal of Symbolic Computation},
 volume={45}, number={12}, pages={1270--1279}, year={2010}
}

@article{deKlerkVallentin16,
 author={de Klerk, Etienne and Vallentin, Frank},
 title={On the {Turing} Model Complexity of Interior Point Methods for Semidefinite Programming},
 journal={SIAM Journal on Optimization}, volume={26}, number={3}, pages={1944--1961}, year={2016}
}

@article{PorkolabKhachiyan97,
 author={Porkolab, Lorant and Khachiyan, Leonid},
 title={On the Complexity of Semidefinite Programs},
 journal={Journal of Global Optimization}, volume={10}, number={4}, pages={351--365}, year={1997}
}

@article{PatakiTouzov24,
 author={Pataki, G{\'a}bor and Touzov, Aleksandr},
 title={How Do Exponential Size Solutions Arise in Semidefinite Programming?},
 journal={SIAM Journal on Optimization}, volume={34}, number={1}, year={2024}, pages={977--1005}
}

@article{TarasovVyalyi08,
 author={Tarasov, Sergey P. and Vyalyi, Mikhail N.},
 title={Semidefinite Programming and Arithmetic Circuit Evaluation},
 journal={Discrete Applied Mathematics}, volume={156}, number={11}, pages={2070--2078}, year={2008}
}

@article{AllenderEtAl09,
 author={Allender, Eric and B{\"u}rgisser, Peter and Kjeldgaard-Pedersen, Johan and Miltersen, Peter Bro},
 title={On the Complexity of Numerical Analysis},
 journal={SIAM Journal on Computing}, volume={38}, number={5}, pages={1987--2006}, year={2009}
}

@article{SaundersonEtAl12,
  author  = {Saunderson, James and Chandrasekaran, Venkat and
             Parrilo, Pablo A. and Willsky, Alan S.},
  title   = {Diagonal and Low-Rank Matrix Decompositions,
             Correlation Matrices, and Ellipsoid Fitting},
  journal = {SIAM Journal on Matrix Analysis and Applications},
  volume  = {33},
  number  = {4},
  pages   = {1395--1416},
  year    = {2012}
}

@article{Laurent00Completion,
  author  = {Laurent, Monique},
  title   = {Polynomial Instances of the Positive Semidefinite
             and Euclidean Distance Matrix Completion Problems},
  journal = {SIAM Journal on Matrix Analysis and Applications},
  volume  = {22},
  number  = {3},
  pages   = {874--894},
  year    = {2000}
}

@article{CrescenziSilvestriTrevisan01,
  author = {Crescenzi, Pierluigi and Silvestri, Riccardo and Trevisan, Luca},
  title = {On Weighted vs Unweighted Versions of Combinatorial Optimization Problems},
  journal = {Information and Computation},
  volume = {167},
  number = {1},
  pages = {10--26},
  year = {2001}
}

\end{document}